\documentclass[12pt]{article}
\usepackage{latexsym, amssymb, amsmath, amscd, amsfonts, epsfig, graphicx, colordvi,verbatim,ifpdf}
\usepackage{amsfonts, amsmath, amssymb}
\usepackage{amssymb,amsfonts,amsmath,latexsym,epsfig,cite, psfrag,eepic,color}
\usepackage{amscd,graphics}
\usepackage{latexsym, amssymb,  amsmath,amscd, amsfonts, epsfig, graphicx, colordvi,amsthm}
\usepackage{color}

\usepackage{graphicx}
\usepackage{color}
\usepackage{ifpdf}
\usepackage{fancybox}
\usepackage[colorlinks,urlcolor=purple,linkcolor=red,anchorcolor=green,citecolor=blue]{hyperref}
\usepackage{mathrsfs}
\usepackage{mathtools}
\usepackage{enumerate}
\usepackage{amsmath}

\def\qed{\nopagebreak\hfill{\rule{4pt}{7pt}}
\medbreak}

\def\p{{\overline{p}}}

\newlength{\boxedparwidth}
  {\begin{center} \begin{tabular}{|@{\hspace{.315in}}c@{\hspace{.15in}}|}
                  \hline \\ \begin{minipage}[t]{\boxedparwidth}
                  \setlength{\parindent}{.25in}}%
  {\end{minipage} \\ \\ \hline \end{tabular} \end{center}}

\allowdisplaybreaks

\newtheorem{thm}{Theorem}[section]

\newtheorem{lem}[thm]{Lemma}
\newtheorem{cor}[thm]{Corollary}

\makeatletter

\newcommand{\Rmnum}[1]{\expandafter\@slowromancap\romannumeral #1@}
\makeatother

\numberwithin{equation}{section}

\begin{document}

\baselineskip 20pt

\newcommand{\lr}[1]{\langle #1 \rangle}
\newcommand{\llr}[1]{\langle\hspace{-2.5pt}\langle #1
\rangle\hspace{-2.5pt}\rangle}

\begin{center}

 {\Large \bf Inequalities for the $k$-Regular Overpartitions}
\end{center}
\vskip 0.2cm

\begin{center}
{Yi Peng}$^{1}$, {Helen W.J. Zhang}$^{1, 2}$ and
  {Ying Zhong}$^{1}$ \vskip 2mm

$^{1}$School of Mathematics\\[2pt]
Hunan University, Changsha 410082, P.R. China\\[5pt]

$^2$Hunan Provincial Key Laboratory of \\ Intelligent Information Processing and Applied Mathematics,
\\[2pt] Changsha 410082, P.R. China\\[5pt]

Email: iampenny@hnu.edu.cn, \quad helenzhang@hnu.edu.cn, \quad YingZhong@hnu.edu.cn
\end{center}

\date{}
\vskip 6mm \noindent {\bf Abstract.}
Bessenrodt and Ono, Chen, Wang and Jia, DeSalvo and Pak were the first to discover the log-subadditivity, log-concavity, and the third-order Tur\'{a}n inequality of partition function, respectively.
Many other important partition statistics are proved to enjoy
similar properties.
This paper focuses on the partition function $\p_k(n)$, which counts the number of overpartitions of $n$ with no parts divisible by $k$.
We provide a combinatorial proof to establish that for any $k\geq2$, the partition function $\p_k(n)$ exhibits strict log-subadditivity. Specifically, we show that $\p_k(a)\p_k(b)>\p_k(a+b)$ for integers $a\geq b\geq1$ and $a+b\geq k$.
Furthermore, we investigate the log-concavity and the satisfaction of the third-order Tur\'{a}n inequality for $\p_k(n)$, where $2\leq k\leq9$.
\vskip 0.3cm

  \vskip 0.3cm

  \vskip 0.3cm

\section{Introduction}
The aim of this paper is to present certain inequality relations satisfied by $k$-regular overpartition function, specifically focusing on strict log-subadditivity, log-concavity, and the third-order Tur\'{a}n inequality.
Recall that a partition of a positive integer $n$ is a non-increasing sequence of positive integers whose sum is $n$.
Let $p(n)$ denote the number of partitions of $n$.
Bessenrodt and Ono \cite{Bessenrodt-Ono-2016} showed that $p(n)$ satisfies the strict log-subadditivity result, that is, for $a,b>1$ and $a+b>9$
\begin{align}\label{B-O}
p(a)p(b)>p(a+b).
\end{align}
Alanazi, Gagola and Munagi \cite{A-A-M-COMBI} resolved the conjecture proposed by Bessenrodt and Ono \cite{Bessenrodt-Ono-2016}, providing a combinatorial proof for equation \eqref{B-O}.
Since then, various partition functions have been proven to enjoy the same strict log-subadditivity as $p(n)$,
for instance, the overpartition function \cite{{Li-2023},{Liu-Zhang-2021}},
the $k$-regular partition function \cite{Beckwith-Bessenrodt-2016} and the spt-function \cite{Chen-2017,Dawsey-Masri-2019}.

Recall that an overpartition of a nonnegative integer $n$ is a
partition of $n$ where the first occurrence of each distinct part may be overlined.
In \cite{Lovejoy-2003}, Lovejoy initiated the exploration of an overpartition function, characterized by its count of overpartitions of $n$ with parts that are indivisible by $k$.
Subsequently, Shen \cite{shen-l-re-overpar} reinterpreted this function, giving it the appellation of $k$-regular overpartition $\p_k(n)$, and elucidated its generating function as follows.
\begin{align}\label{gene}
\sum_{n=0}^\infty \p_k(n)q^n=\frac{(q^k;q^k)_\infty(-q;q)_\infty}{(q;q)_\infty(-q^k;q^k)_\infty},
\end{align}
where $(a;q)_\infty$ stands for the $q$-shifted factorial
\[(a;q)_\infty=\prod_{n=1}^{\infty}(1-aq^{n-1}),\quad |q|<1.\]

The first result of this paper is the demonstration, using combinatorial methods, that $\p_k(n)$ satisfies strict log-subadditivity.
\begin{thm}\label{thm1}
For integers $k$, $a$ and $b$, if $k\geq 2$, $a\geq b\geq 1$, $a+b\geq k$, then
\begin{align*}
\p_k(a)\p_k(b)>\p_k(a+b).
\end{align*}
\end{thm}

Having set the stage with our first result, demonstrating the strict log-subadditivity of $\p_k(n)$ through combinatorial methods, we now transition towards a broader context. We delve into the realm of log-concavity, also known as Tur\'{a}n inequalities, along with their higher-order counterparts, known as cubic inequalities, originate from the examination of the Maclaurin coefficients of real entire functions within the Laguerre-P\'{o}lya class.
A sequence $\{\alpha_n\}_{n\geq0}$ is considered to be log-concave if it fulfills the Tur\'{a}n inequalities, which are given by
\[\alpha_n^2\geq \alpha_{n-1}\alpha_{n+1},~~\text{for}~n\geq 1.\]
Moreover, a sequence $\{\alpha_n\}_{n\geq0}$ is defined as satisfying the third order Tur\'{a}n inequalities if for $n\geq1$,
\begin{align*}
4(\alpha_n^2-\alpha_{n-1}\alpha_{n+1})
(\alpha_{n+1}^2-\alpha_{n}\alpha_{n+2})
\geq(\alpha_n\alpha_{n+1}-\alpha_{n-1}\alpha_{n+2})^2.
\end{align*}

The initial exploration of the Tur\'{a}n inequalities, along with their higher-order variants for the partition function, can be credited to the research undertaken by Chen, Jia and Wang \cite{turan-for-partition}, DeSalvo and Pak \cite{DeSalvo-Pak-2015}.
In the wake of these seminal contributions, a substantial body of related research has burgeoned. This includes noteworthy contributions from scholars such as Bringmann, Kane, Rolen, and Tripp \cite{bringmann-i-bessel}, Dong and Ji \cite{dong-ji}, Engel \cite{Engel-2017}, Griffin, Ono, Rolen, and Zagier \cite{Griffin-Ono-Rolen-Zagier-2019}, Jia \cite{Jia-2023}, Liu and Zhang \cite{Liu-Zhang-2021}, and Ono, Pujahari, and Rolen \cite{Ono-Ken-Rolen-2022}.
In furthering this line of research, our work reveals new insights about the properties of $k$-regular overpartitions. We establish that when $2\leq k\leq9$, these overpartitions fulfill the conditions of log-concavity and adhere to the third-order Tur\'{a}n inequalities.

In order to establish these inequality relations, our initial undertaking involves developing the asymptotic formulae for $\p_k(n)$, specifically when $2\leq k\leq9$. This critical step in our methodology ultimately allows us to unveil the stated inequality relations.
For convenience, we adopt the notation below, set
\[{\mu_2(n)}=\frac{\pi }{2}\sqrt{2n},\quad\mu_3(n)=\frac{\pi}{3}\sqrt{6n},
\quad\mu_4(n)=\frac{\pi}{2}\sqrt{3n},
\quad\mu_5(n)=\frac{2\pi}{5}\sqrt{5n},\]
\[\quad\mu_6(n)=\frac{\pi}{6}\sqrt{30n},
\quad\mu_7(n)=\frac{\pi}{7}\sqrt{42n},
\quad\mu_8(n)=\frac{\pi}{4}\sqrt{14n},
\quad\mu_9(n)=\frac{2\pi}{3}\sqrt{2n}.\]

For the sake of convenience, we use the following notations throughout this paper:
\[{\mu_k^-}=\mu_k(n-1),\quad {\mu_k}=\mu_k(n),\quad{\mu_k^+}=\mu_k(n+1).\]
\begin{thm}\label{thm2}
Let $I_1(s)$ denote the first modified Bessel function of the first kind.
For each fixed $2\leq k\leq9$, and $\mu_k\geq n_k$, we have
\begin{align*}
\p_k(n)=C_k(n)I_1({\mu_k})+R_k(n),
\end{align*}
where
\[|R_k(n)|\leq R_k'(n),\]
and all these values are defined in Table \ref{pk(n)-bounds}.
\begin{table}[h]
  \centering
\begin{tabular}{c|c|c|c||c|c|c|c}
  $k$ & $n_k$ & $C_k(n)$ & $R_k'(n)$ &$k$ & $n_k$ & $C_k(n)$ & $R_k'(n)$
  \\[3pt]   \hline \rule{0pt}{23pt}
  $2$ & $22$
  & {\large$\frac{\pi^2}{\sqrt8{\mu_2}}$}
  & {\large$\frac{\sqrt3\pi^{\frac{3}{2}}}{2\sqrt{\mu_2}}$}
  $\exp\left(\frac{{\mu_2}}{3}\right)$
  &$6$ & $130$
  & {\large$\frac{5\sqrt6 \pi^2}{18\mu_6}$}
  & {\large$\frac{10\sqrt{15}\pi^{\frac{3}{2}}}{27\sqrt{\mu_6}}$}
  $\exp\left(\frac{\mu_6}{5}\right)$
  \\[9pt] \hline \rule{0pt}{23pt}
  $3$ & $49$
  & {\large$\frac{2\sqrt3 \pi^2}{9\mu_3}$}
  & {\large$\frac{4\sqrt{30}\pi^{\frac{3}{2}}}{27\sqrt{\mu_3}}$}
  $\exp\left(\frac{\mu_3}{5}\right)$
  &$7$ & $102$
  & {\large$\frac{18\sqrt7\pi^2}{49\mu_7}$}
  & {\large$\frac{108\sqrt{42}\pi^{\frac{3}{2}}}{343\sqrt{\mu_7}}$}
  $\exp\left(\frac{\mu_7}{3}\right)$
  \\[9pt] \hline \rule{0pt}{23pt}
  $4$ & $41$
  & {\large$\frac{3 \pi^2}{4\mu_4}$}
  & {\large$\frac{3\sqrt{6}\pi^{\frac{3}{2}}}{4\sqrt{\mu_4}}$}
  $\exp\left(\frac{\mu_4}{3}\right)$
  &$8$ & $129$
  & {\large$\frac{7\sqrt2\pi^2}{8\mu_8}$}
  & {\large$\frac{7\sqrt3\pi^{\frac{3}{2}}}{4\sqrt{\mu_8}}$}
  $\exp\left(\frac{\mu_8}{3}\right)$
  \\[9pt] \hline \rule{0pt}{23pt}
  $5$ & $58$
  & {\large$\frac{8\sqrt5 \pi^2}{25\mu_5}$}
  & {\large$\frac{32\sqrt{30}\pi^{\frac{3}{2}}}{125\sqrt{\mu_5}}$}
  $\exp\left(\frac{\mu_5}{3}\right)$
  &$9$ & $268$
  & {\large$\frac{8\pi^2}{9\mu_9}$}
  & {\large$\frac{16\sqrt{10}\pi^{\frac{3}{2}}}{27\sqrt{\mu_9}}$}
  $\exp\left(\frac{\mu_9}{5}\right)$
  \\[3pt]
\end{tabular}
\caption{Values of $n_k$, $C_k(n)$ and $R_k'(n)$ for $2\leq k\leq9$.}
\label{pk(n)-bounds}
\end{table}
\end{thm}

\begin{thm}\label{thm1.5}
For $2\leq k\leq9$, $\p_k(n)$ satisfies log-concavity for $n\geq \bar{n}_k$ and the third-order Tur\'{a}n inequality for $n\geq \hat{n}_k$.
The specific ranges are provided in the following table.
\begin{center}
\renewcommand{\arraystretch}{1.5}
\begin{tabular}{p{1cm}|p{1cm}|p{1cm}|p{1cm}|p{1cm}|p{1cm}|p{1cm}|p{1cm}|p{1cm}}
\centering $k$ & \centering$2$ & \centering $3$ &  \centering$4$ & \centering$5$ & \centering$6$  & \centering$7$ & \centering$8$ & $\quad9$  \\
\hline
\centering$\bar{n}_k$ & \centering $21$ & \centering $4$ & \centering$5$ & \centering$6$ & \centering$1$ & \centering$1$ & \centering$1$ & $\quad1$\\
\hline
\centering$\hat{n}_k$ & \centering $65$  & \centering $23$   & \centering$28$ & \centering$26$ & \centering$11$ & \centering$22$ & \centering$23$ & $~~10$\\
\end{tabular}
\end{center}
\end{thm}

The paper is organized as follows.
In Section \ref{Sec-2}, we rigorously prove the strict log-subadditivity of $\p_k(n)$ using combinatorial methods.
Section \ref{Sec-3} utilizes the results established by Chern, Ji and Dong to derive the asymptotic formulas for $\p_k(n)$ when $2\leq k\leq9$.
Lastly, in Section \ref{Sec-4}, the derived asymptotic formula is applied to demonstrate the log-concavity and third-order Tur\'{a}n inequalities of $\p_k(n)$ for $2\leq k\leq9$.

\section{Proof of Theorem \ref{thm1}}\label{Sec-2}
In this section, we establish a series of lemmas to support the proof. Throughout this section, we employ combinatorial methods to demonstrate the log-subadditivity property of $\p_k(n)$, considering that exponents refer to multiplicities.
\begin{lem}\label{lem1}
For an integer $k$ with $k\geq2$, if $a$, $b$ are integers with $a,b\geq 1$, then
\[\p_k(a|\ {\rm no\ 1^{\prime}s})\p_k(b|\ {\rm no\ 2^{\prime}s})\geq\p_k(a+b|\ {\rm no\ 1's\ and\ no \ 2's}).\]
\end{lem}

\begin{proof}
For $\lambda=(\lambda_1,\dots,\lambda_t)\in \overline P_k(a+b)$, we define \[i=i(\lambda)=\mathrm max\{j\in \mathbb{N}|1\leq j \leq t,\lambda_j+\dots+\lambda_t \geq b\}.\]
Let $\lambda_i=x+y$ such that
\[y+\lambda_1+\dots+\lambda_{i-1}=a ~~\text{and}~~ x+\lambda_{i+1}+\dots+\lambda_t=b,\]
where $x\geq1$.
With these definitions, we introduce the map
\[f_1:\ \overline P_k(a+b|\text{ no 1's and no 2's})\rightarrow \overline P_k(a|\text{ no 1's})\oplus \overline P_k(b|\text{ no 2's}).\]
Next, we aim to establish the injectivity of $f_1$ by examining the parity of $k$.
For $\lambda=(\lambda_1,\dots,\lambda_t)\in \overline P_k(a+b|\ {\rm no\ 1's\ and\ no \ 2's})$.
When $k$ is odd, we consider the case of $k=2m+1$ ($m\geq2$), specifically focusing on the situation when $k\geq5$. The proof method for $k=3$ follows a similar approach.
We will divide our analysis into six cases for further discussion.
\begin{enumerate}[1.]
\item \label{1} If $y=0$, we set
\begin{align*}
f_1(\lambda)=(\lambda_1,\dots,\lambda_{i-1};\lambda_i,\dots,\lambda_t).
\end{align*}
\item \label{2} If $y=1$, we consider two separate cases by distinguishing whether $\lambda_i$ is overlined.

{\rm(1).} When $\lambda_i$ is overlined, we have
$f_1(\lambda)=(\lambda_1,\dots,\lambda_{i-1},\overline 1;\lambda_{i+1},\dots,\lambda_t,1^x).$

{\rm(2).} When $\lambda_i$ is non-overlined and $\lambda_{i-1}$ is overlined,
it leads to the following expression.
\begin{align*}
f_1(\lambda)=
\begin{dcases}
(\lambda_1,\dots,\overline{\lambda_{i-1}-k},2^{m+1},
\lambda_{i+1},\dots,\lambda_t,1^x),&\lambda_{i-1}\geq {k+2},
\\[3pt]
(\lambda_1,\dots,\lambda_{i-2},\overline2,2^{c};\lambda_{i+1},\dots,\lambda_t,1^x),
&3\leq \lambda_{i-1}=2c+1\leq{k+1},
\\[3pt]
(\lambda_1,\dots,\lambda_{i-2},\overline2,2^{c-1},\overline1;
\lambda_{i+1},\dots,\lambda_t,1^x),&3\leq \lambda_{i-1}=2c\leq{k+1}.
\end{dcases}
\end{align*}

{\rm(3).} When both $\lambda_i$ and $\lambda_{i-1}$ are non-overlined, we observe the following mapping.
\begin{align*}
f_1(\lambda)=
\begin{dcases}
(\lambda_1,\dots,\lambda_{i-1}-k,2^{m+1};\lambda_{i+1},\dots,\lambda_t,1^x),&\lambda_{i-1}\geq {k+2},
\\
(\lambda_1,\dots,\lambda_{i-2},2^{c+1};\lambda_{i+1},\dots,\lambda_t,1^x),&3\leq \lambda_{i-1}=2c+1\leq{k+1},
\\
(\lambda_1,\dots,\lambda_{i-2},2^{c},\overline1;\lambda_{i+1},\dots,\lambda_t,1^x),&3\leq \lambda_{i-1}=2c\leq{k+1}.
\end{dcases}
\end{align*}
\item \label{3}
   If $y\equiv i\pmod{k},\ (i=2,\dots,k-1)$, we have $y\geq i$, which leads to the following mapping.
   \begin{align*}
   f_1(\lambda)=\begin{cases}
   (\lambda_1,\dots,\lambda_{i-1},\overline y;\lambda_{i+1},\dots,\lambda_t,1^x),&\lambda_i\ \mathrm{is\ overlined},
   \\[5pt]
   (\lambda_1,\dots,\lambda_{i-1},y;\lambda_{i+1},\dots,\lambda_t,1^x),&\lambda_i\ \mathrm{is}\ \mathrm{non\mbox{-}overlined}.
   \end{cases}
   \end{align*}
\item \label{4}
   If $y\equiv1\pmod{k}$ and $y\geq{k+1}$, we obtain the following mapping.
   \begin{align*}
   f_1(\lambda)=\begin{cases}(\lambda_1,\dots,\lambda_{i-1},\overline y;\lambda_{i+1},\dots,\lambda_t,1^x),&\lambda_i\ \mathrm{is\ overlined},
   \\[5pt]
   (\lambda_1,\dots,\lambda_{i-1},y;\lambda_{i+1},\dots,\lambda_t,1^x),&\lambda_i\ \mathrm{is}\ \mathrm{non\mbox{-}overlined}.\end{cases}
   \end{align*}
\item \label {5}
   If $y=k$, we represent the mapping in the following form.
   \begin{align*}
   f_1(\lambda)=\begin{cases}
   (\lambda_1,\dots,\lambda_{i-1},
   \overline{m+1},m;\lambda_{i+1},\dots,\lambda_t,1^x),
   &\lambda_i\geq{k+2} \  \mathrm{and\ overlined},
   \\
   (\lambda_1,\dots,\lambda_{i-1},{m+1},m;
   \lambda_{i+1},\dots,\lambda_t,1^x),
   &\lambda_i\geq{k+2}\ \mathrm{and\ non\mbox{-}overlined},
   \\
   (\lambda_1,\dots,\lambda_{i-1},\overline2,2^{m-1},\overline1;
   \lambda_{i+1},\dots,\lambda_t,1),
   &\lambda_i=k+1\ \mathrm{and\ overlined},
   \\
   (\lambda_1,\dots,\lambda_{i-1},2^m,\overline1;
   \lambda_{i+1},\dots,\lambda_t,1),
   &\lambda_i=k+1\ \mathrm{and\ non\mbox{-}overlined}.\end{cases}
\end{align*}
\item \label {6}
  If $y\equiv0\pmod{k}$ and $y\geq{2k}$, we divide our analysis into two distinct cases based on whether $\lambda_i$ is overlined or not.

  {\rm(1).} When $\lambda_i$ is overlined and $\lambda_i\equiv j\pmod{k}$, we observe that
  \begin{align*}
  f_1(\lambda)=\begin{dcases}
  (\lambda_1,\dots,\lambda_{i-1},\overline{y-(k-j)},k-j;
  \lambda_{i+1},\dots,\lambda_t,1^x),
  &1\leq j\leq k-2,
  \\[5pt]
  (\lambda_1,\dots,\lambda_{i-1},\overline{y-1},\overline1;
  \lambda_{i+1},\dots,\lambda_t,1^x),
  &j=k-1.
\end{dcases}
\end{align*}

  {\rm(2).} When $\lambda_i$ is non-overlined and $\lambda_i\equiv j\pmod{k}$, the following observation becomes apparent.
  \begin{align*}
  f_1(\lambda)=\begin{dcases}
  (\lambda_1,\dots,\lambda_{i-1},{y-(k-j)},k-j;\lambda_{i+1},\dots,\lambda_t,1^x),
  &1\leq j\leq k-2,
  \\[5pt]
  (\lambda_1,\dots,\lambda_{i-1},{y-1},\overline1;
  \lambda_{i+1},\dots,\lambda_t,1^x),
  &j=k-1.
\end{dcases}
\end{align*}
\end{enumerate}

Notice that in cases \ref{3} and \ref{4}, $y\geq 2$ and $2$ (or $\overline2$) as a part appear at most once, which has no intersection with case \ref{2}.
In case \ref{5}, if $\lambda_i={k+1}$, then $x=1$.
However, in case \ref{2} when $\lambda_i$ is non-overlined, we have $\lambda_i\geq3$ and $x\geq2$, which implies there is no intersection between the two cases.
Similarly, in case \ref{2} where $\lambda_i$ is overlined, if $x=1$, then $\lambda_i=y+x=\overline2$.
According to the principles of notations, $\lambda_{i-1}>\lambda_i=\overline2$, which means the appearance of two is zero.
Thus, case \ref{5} has no intersection with case \ref{2} either.
By analogy, case \ref{5} shares no intersection with cases \ref{3} and \ref{4}.
In case \ref{6}, where $j\neq{k-1}$, the multiplicity of two is at most one. When $j=k-1$, we have $x\equiv {k-1}\pmod{k}$.
However, in case \ref{2}, when $\lambda_i$ is overlined and $x\equiv {k-1}\pmod{k}$, we have $\lambda_i=y+x=1+x \equiv0\pmod{k}$, which implies that $k$ divides $\lambda_i$, a contradiction.
Hence, case \ref{6} is disjoint from case \ref{2}.
The same reasoning applies when comparing case \ref{6} with cases \ref{3} and \ref{4}.
Regarding the comparison between case \ref{5} and case \ref{6}, we note that in case \ref{6}, the difference between the last two parts of the partition, namely $y-(k-j)$ and $k-j$, is greater than $1$.
Therefore, when $k$ is odd, $f_1$ is one-to-one.

When $k$ is even, let $k=2m,$ here we present the cases below for $m\geq3$, that is $k\geq6$. The proof approaches to the cases for $k=2$ and $k=4$ are similar to $k\geq6$.
\[f_1'(\lambda)=\begin{dcases}(\lambda_1,\dots,\overline{\lambda_{i-1}-k},2^{m},\overline 1;\lambda_{i+1},\dots,\lambda_t,1^x),&y=1,\lambda_i\ \mathrm{is\ non \mbox{-}overlined},\lambda_{i-1}\geq {k+2},\\ &\lambda_{i-1}\ \mathrm{is\ overlined};\\
(\lambda_1,\dots,\lambda_{i-1}-k,2^{m},\overline 1;\lambda_{i+1},\dots,\lambda_t,1^x),&y=1,\lambda_i\ \mathrm{is\ non\mbox{-}overlined},\lambda_{i-1}\geq {k+2},\\ &\lambda_{i-1}\ \mathrm{is\ non\mbox{-}overlined};\\
(\lambda_1,\dots,\lambda_{i-1},\overline{m},m;\lambda_{i+1},\dots,\lambda_t,1^x),&y=k,\lambda_i\ \mathrm{is\ overlined};\\
(\lambda_1,\dots,\lambda_{i-1},m^2;\lambda_{i+1},\dots,\lambda_t,1^x),&y=k,\lambda_i\ \mathrm{is\ non\mbox{-}overlined};\\
f_1(\lambda),&\mathrm{else}.\end{dcases}
\]
With minor adjustments, it is obvious to know $f_1'$ is also one-to-one.
\par Therefore, through all cases above, we can always construct a map that is one-to-one for $k\geq2$ to meet the inequality. Hence, the proof for the inequalities is completed.
\end{proof}
\begin{lem}\label{lem2}
If $k$, $a$ are integers with $k\geq2$, $a\geq1$, then\[\p_k(a|\ \mathrm{no\ 2's})\p_k(1)>\p_k(a+1|\ \mathrm{no\ 2's}).\]
\end{lem}
\begin{proof}
Let $\lambda=(\lambda_1,\dots,\lambda_t,{\overline{1}}^r,1^s)\in \overline P_k(a+1|{\rm no\ 2's})$ with $r\in\{0,1\}$.
We introduce the mapping
\[f_2: \overline P_k(a+1|\ \mathrm{no\ 2's})\rightarrow
\overline P_k(a|\ \mathrm{no\ 2's})\oplus\overline P_k(1)\]
as
\begin{align*}
f_2(\lambda)=\begin{cases}
(\lambda_1,\dots,\lambda_t,\overline 1^r,1^{s-1};1),&{\rm if}\ s\geq1,
\\
(\lambda_1,\dots,\lambda_{t-1},1^{\lambda_t-1};\overline1),&{\rm if}\ s=0,r=0,\lambda_t\ {\rm is\ non\mbox{-}overlined},
\\
(\lambda_1,\dots,\lambda_{t-1},\overline1,1^{\lambda_t-2};\overline1),&{\rm if}\ s=0,r=0,\lambda_t\ {\rm is\ overlined},
\\
(\lambda_1,\dots,\lambda_t;\overline1),&{\rm if}\ s=0,r=1.\end{cases}
\end{align*}

The mapping $f_2$ is well-defined and injective. However, it is evident that $(\lambda_1,\dots,\lambda_t,1;\overline1)$, with $\lambda_t>1$, is not in the image of $f_2$. Therefore, the inequality holds, which completes the proof.
\end{proof}

\begin{lem}\label{lem3}
If $k$, $a$ are integers with $k\geq2$, $a\geq1$, then\[\p_k(a|\ \mathrm{no\ 2's})\p_k(2)>\p_k(a+2|\ \mathrm{no\ 2's}).\]
\end{lem}
\begin{proof}
Define a map
\[f_3: \overline P_k(a+2|\ \mathrm{no\ 2's})\rightarrow \overline P_k(a|\ \mathrm{no\ 2's})\oplus\overline P_k(2)\]
by
\begin{align*}
f_3(\lambda)=\begin{cases}
(\lambda_1,\dots,\lambda_t,\overline 1^r,1^{s-2};2),&{\rm if}\ s\geq2,
\\
(\lambda_1,\dots,\lambda_t;\overline2),&{\rm if}\ s=1,~r=1,
\\
(\lambda_1,\dots,\lambda_{t-1},\overline1,1^{\lambda_t-2};1^2),&{\rm if}\ s=1,~r=0,~\lambda_t\ {\rm is\ overlined},
\\
(\lambda_1,\dots,\lambda_{t-1},1^{\lambda_t-1};1^2),&{\rm if}\ s=1,~r=0,~\lambda_t\ {\rm is\ non\mbox{-}overlined},
\\
(\lambda_1,\dots,\lambda_{t-1};1^2),
&{\rm if}\ s=0,~r=0,~\lambda_t=\overline2,
\\
(\lambda_1,\dots,\lambda_{t-1},\overline1,1^{\lambda_t-3};\overline1,1),
&{\rm if}\ s=0,~r=0,~\lambda_t\geq3,\ {\rm is\ overlined},
\\
(\lambda_1,\dots,\lambda_{t-1},1^{\lambda_t-2};\overline1,1),
&{\rm if}\ s=1,~r=0,~\lambda_t\geq3,\ {\rm is\ non\mbox{-}overlined},
\\
(\lambda_1,\dots,\lambda_{t-1},\overline1,1^{\lambda_t-2};\overline2),
&{\rm if}\ s=0,~r=1,~\lambda_t\ {\rm is\ overlined},
\\
(\lambda_1,\dots,\lambda_{t-1},1^{\lambda_t-1};\overline2),
&{\rm if}\ s=0,~r=1,~\lambda_t\ {\rm is\ non\mbox{-}overlined}.
\end{cases}
\end{align*}

The mapping $f_3$ is well-defined and injective. It is clear that $(\lambda_1,\dots,\lambda_{t};\overline1,1)$, with $\lambda_t>1$, is not in the image of $f_3$. This completes the proof.
\end{proof}

\begin{lem}\label{lem4}
If $k$, $a$, $b$, are integers with $k\geq2$, $a\geq1$, $b\geq3$  and $a+b\geq{k+1}$, then\[\p_k(a|\ \mathrm{no\ 2's})\p_k(b)>\p_k(a+b|\ \mathrm{no\ 2's}).\]
\end{lem}
\begin{proof}
Let $n=a+b$. We will apply induction on $n$ to prove the inequality.

Base cases: The inequality holds for $n=k+1$. In this case, we have $a+b=k+1$, where $b\geq3$ and $a\geq1$. It follows that $a\leq k-2$ and $b\leq k$.
When $a,b<k$, we have
\[\p_k(a|\ \mathrm{no\ 2's})\p_k(b)=\p(a|\ \mathrm{no\ 2's})\p(b)\geq \p(a+b|\ \mathrm{no\ 2's})>\p_k(a+b|\ \mathrm{no\ 2's}).\] If $b=k$ and $a=1$, we can construct a mapping to show that the result holds.

Inductive step: Suppose $n\geq{k+2}$ and the inequality holds for $n-1$. Now, if $b\geq4$, then $\p_k(a|\ \mathrm{no\ 2's})\p_k(b-1)>\p_k(a+b-1|\ \mathrm{no\ 2's})$ by the inductive hypothesis.
Similarly, if $b=3$, we can apply Lemma \ref{lem3} to conclude that
\[\p_k(a|\ \mathrm{no\ 2's})\p_k(b-1)>\p_k(a+b-1|\ \mathrm{no\ 2's}).\]
Using Lemma \ref{lem1}, we have:
\begin{align*}
\p_k(a+b|\ {\rm no\ 2's})&= \p_k(a+b|\ {\rm no\ 1's\ and \ no\ 2's})+\p_k(a+b-1|\ {\rm no\ 2's})\\
&\leq \p_k(a|\ {\rm no\ 2's})\p_k(b|\ {\rm no\ 1's})+\p_k(a+b-1|\ {\rm no\ 2's})\\
&< \p_k(a|\ {\rm no\ 2's})(\p_k(b|\ {\rm no\ 1's})+\p_k(b-1))\\
&= \p_k(a|\ {\rm no\ 2's})\p_k(b).
\end{align*}
Therefore, by the principle of mathematical induction, the inequality holds.
\end{proof}

Based on the preceding lemmas and employing mathematical induction, we are ready to commence the proof of Theorem \ref{thm1}.

{\noindent{\it Proof of Theorem \ref{thm1}}.}
Let $n=a+b$. We apply induction on $n$.

Base cases: When $n=k,\ k+1$, it holds.
If $n=k$, then $a+b=k$ and $a,b<k$, we have \[\p_k(a)\p_k(b)=\p(a)\p(b)>\p(a+b)>\p_k(a+b).\]
If $a+b=k+1$, when $a,b<k$, it can be proved in a similar way.
By symmetry, set $a\geq b$, when $a=k$, $b=1$, the case belongs to Lemma \ref{lem2} that $a=k$, so the inequality is established.

Inductive step: Suppose $n\geq{k+2}$ and the inequality holds for $n-2$. 
Therefore, by the inductive hypothesis, $\p_k(a+b-2)<\p_k(a-2)\p_k(b)$. By Lemmas \ref{lem2}-\ref{lem4},
\[\p_k(a+b|\ {\rm no\ 2's})<\p_k(a|\ {\rm no\ 2's})\p_k(b).\]
Hence,
\begin{align*}
\p_k(a+b)&=\p_k(a+b|\ {\rm no\ 2's})+\p_k(a+b-2)\\
&< \p_k(a|\ {\rm no\ 2's})\p_k(b)+\p_k(a-2)\p_k(b)\\
&=(\p_k(a|\ {\rm no\ 2's})+\p_k(a-2))\p_k(b)\\
&= \p_k(a)\p_k(b).
\end{align*}
Thus, by the principle of mathematical induction, the inequality holds.
\qed

\section{Proof of Theorem \ref{thm2}}\label{Sec-3}

In this section, our aim is to establish the asymptotic formula for $\p_k(n)$ when $2 \leq k \leq 9$.
To accomplish this,
we first utilize the following lemma given by Ji and Dong \cite[Lemma~2.2]{dong-ji} to obtain inequalities involving $I_1({\mu_k(n-1)})I_1({\mu_k(n+1)})/I_1({\mu_k(n)})^2$,
which are necessary for the subsequent proofs.

\begin{lem}\label{lem3.2}
Let \begin{align*}
E_I(s):=1-\frac{3}{8s}-\frac{15}{128s^2}-\frac{105}{1024s^3}
-\frac{4725}{32768s^4}-\frac{72765}{262144s^5}.
\end{align*}
Then for $s\geq26$,
\begin{align}\label{3.3}
\frac{e^s}{\sqrt{2\pi s}}\left(E_I(s)-\frac{31}{s^6}\right)\leq I_1(s)\leq \frac{e^s}{\sqrt{2\pi s}}\left(E_I(s)+\frac{31}{s^6}\right).
\end{align}
\end{lem}

\begin{lem}\label{lem3.3}
For $2\leq k\leq9$ and ${\mu_k}\geq n'_k$,
\begin{align*}
\frac{{\mu_k}}{\sqrt{{\mu_k^-}{\mu_k^+}}}U_k(n)
\leq\frac{I_1({\mu_k^-})I_1({\mu_k^+})}{I_1({\mu_k})^2}\leq
\frac{{\mu_k}}{\sqrt{{\mu_k^-}{\mu_k^+}}}V_k(n),
\end{align*}
where all these values are defined in following table.
\begin{table}[h]
\begin{tabular}{p{0.3cm}|p{0.5cm}|p{6.9cm}|p{6.9cm}}
  \centering$k$ & \centering$n_k'$ & \centering $U_k(n)$ &\qquad\qquad\qquad\quad $V_k(n)$
  \\[3pt]   \hline \rule{0pt}{23pt}
  $2$ &\centering $167$
  &  $\left(1-\frac{\pi^4}{16{\mu_2}^3}-\frac{5\pi^8}{512{\mu_2}^7}\right)
\left(1-\frac{7}{{\mu_2}^5}-\frac{115}{{\mu_2}^6}\right)$
  & $\left(1-\frac{\pi^4}{16{\mu_2}^3}+\frac{\pi^8}{256{\mu_2}^6}\right)
\left(1-\frac{6}{{\mu_2}^5}+\frac{115}{{\mu_2}^6}\right)$
  \\[9pt] \hline \rule{0pt}{23pt}
  $3$ &\centering $47$
  & $\left(1-\frac{\pi^4}{9{\mu_3}^3}-\frac{5\pi^8}{162{\mu_3}^7}\right)
\left(1-\frac{13}{{\mu_3}^5}-\frac{141}{{\mu_3}^6}\right)$
  & $\left(1-\frac{\pi^4}{9{\mu_3}^3}+\frac{\pi^8}{81{\mu_3}^6}\right)
\left(1-\frac{6}{{\mu_3}^5}+\frac{141}{{\mu_3}^6}\right)$
  \\[9pt] \hline \rule{0pt}{23pt}
  $4$ &\centering $58$
  & $\left(1-\frac{9\pi^4}{64{\mu_4}^3}-\frac{405\pi^8}{8192{\mu_4}^7}\right)
\left(1-\frac{16}{{\mu_4}^5}-\frac{145}{{\mu_4}^6}\right)$
  & $\left(1-\frac{9\pi^4}{64{\mu_4}^3}+\frac{81\pi^8}{4096{\mu_4}^6}\right)
\left(1-\frac{15}{{\mu_4}^5}+\frac{145}{{\mu_4}^6}\right)$
  \\[9pt] \hline \rule{0pt}{23pt}
 $5$ & \centering $66$
  & $\left(1-\frac{4\pi^4}{25{\mu_5}^3}-\frac{8\pi^8}{125{\mu_5}^7}\right)
\left(1-\frac{18}{{\mu_5}^5}-\frac{148}{{\mu_5}^6}\right)$
  & $\left(1-\frac{4\pi^4}{25{\mu_5}^3}+\frac{16\pi^8}{625{\mu_5}^6}\right)
\left(1-\frac{17}{{\mu_5}^5}+\frac{148}{{\mu_5}^6}\right)$
  \\[9pt] \hline \rule{0pt}{23pt}
 $6$ & \centering$47$
  & $\left(1-\frac{25\pi^4}{144{\mu_6}^3}-\frac{3125\pi^8}{41472{\mu_6}^7}\right)
\left(1-\frac{20}{{\mu_6}^5}-\frac{150}{{\mu_6}^6}\right)$
  & $\left(1-\frac{25\pi^4}{144{\mu_6}^3}+\frac{625\pi^8}{20736{\mu_6}^6}\right)
\left(1-\frac{19}{{\mu_6}^5}+\frac{150}{{\mu_6}^6}\right)$
  \\[9pt] \hline \rule{0pt}{23pt}
 $7$ &\centering $51$
  & $\left(1-\frac{9\pi^4}{49{\mu_7}^3}-\frac{25920\pi^8}{307328{\mu_7}^7}\right)
\left(1-\frac{21}{{\mu_7}^5}-\frac{151}{{\mu_7}^6}\right)$
  & $\left(1-\frac{9\pi^4}{49{\mu_7}^3}+\frac{81\pi^8}{2401{\mu_7}^6}\right)
\left(1-\frac{20}{{\mu_7}^5}+\frac{151}{{\mu_7}^6}\right)$
  \\[9pt] \hline \rule{0pt}{23pt}
 $8$ &\centering $300$
  & $\left(1-\frac{98\pi^4}{512{\mu_8}^3}-\frac{48020\pi^8}{524288{\mu_8}^7}\right)
\left(1-\frac{21}{{\mu_8}^5}-\frac{152}{{\mu_8}^6}\right)$
  & $\left(1-\frac{98\pi^4}{512{\mu_8}^3}+\frac{9604\pi^8}{262144{\mu_8}^6}\right)
\left(1-\frac{20}{{\mu_8}^5}+\frac{152}{{\mu_8}^6}\right)$
  \\[9pt] \hline \rule{0pt}{23pt}
 $9$ & \centering$81$
  & $\left(1-\frac{16\pi^4}{81{\mu_9}^3}-\frac{640\pi^8}{6561{\mu_9}^7}\right)
\left(1-\frac{22}{{\mu_9}^5}-\frac{153}{{\mu_9}^6}\right)$
  & $\left(1-\frac{16\pi^4}{81{\mu_9}^3}+\frac{256\pi^8}{6561{\mu_9}^6}\right)
\left(1-\frac{21}{{\mu_9}^5}+\frac{153}{{\mu_9}^6}\right)$
  \\[3pt]
\end{tabular}
\caption{Values of $n'_k$, $U_k(n)$ and $V_k(n)$ for $2\leq k\leq9$.}
\label{I1k(n)-bounds}
\end{table}
\end{lem}

\begin{proof}
By \eqref{3.3}, we find that for $2\leq k\leq9$  and ${\mu_k}\geq26$,
\begin{align}\label{3.6}
\frac{{\mu_k}}{\sqrt{{\mu_k^-}{\mu_k^+}}}
e^{{\mu_k^-}+{\mu_k^+}-2{\mu_k}}U'_k(n)
\leq\frac{I_1({\mu_k^-})I_1({\mu_k^+})}{I_1({\mu_k})^2}\leq
\frac{{\mu_k}}{\sqrt{{\mu_k^-}{\mu_k^+}}}
e^{{\mu_k^-}+{\mu_k^+}-2{\mu_k}}V'_k(n),
\end{align}
where
\begin{align}\label{3.8}
U_k'(n)=\frac{\left(E_I({\mu_k^-})-\frac{31}{{\mu_k^-}^6}\right)
\left(E_I({\mu_k^+})-\frac{31}{{\mu_k^+}^6}\right)}
{\left(E_I({\mu_k})+\frac{31}{{\mu_k}^6}\right)^2},
\end{align}
\begin{align}\label{3.9}
V_k'(n)=\frac{\left(E_I({\mu_k^-})+\frac{31}{{\mu_k^-}^6}\right)
\left(E_I({\mu_k^+})+\frac{31}{{\mu_k^+}^6}\right)}
{\left(E_I({\mu_k})-\frac{31}{{\mu_k}^6}\right)^2}.
\end{align}
We consider the case of $k=2$.
First, we will estimate $\exp({\mu_2^-}+{\mu_2^+}-2{\mu_2})$, $U'_2(n)$ and $V'_2(n)$ in terms of ${\mu_2}$. For ${\mu_2}\geq3$, we observe the following relationships:
\begin{equation}\label{3.100}
{\mu_2^-}=\sqrt{{\mu_2}^2-\frac{\pi^2}{2}},\qquad {\mu_2^+}=\sqrt{{\mu_2}^2+\frac{\pi^2}{2}}.
\end{equation}
Thus we can observe
\begin{align*}
{\mu_2^-}&={\mu_2}-\frac{\pi^2}{4{\mu_2}}
-\frac{\pi^4}{32{\mu_2}^3}-\frac{\pi^6}{128{\mu_2}^5}
-\frac{5\pi^8}{128\times16{\mu_2}^7}+O\left(\frac{1}{{\mu_2}^8}\right),
\\[5pt]
{\mu_2^+}&={\mu_2}+\frac{\pi^2}{4{\mu_2}}
-\frac{\pi^4}{32{\mu_2}^3}+\frac{\pi^6}{128{\mu_2}^5}
-\frac{5\pi^8}{128\times16{\mu_2}^7}+O\left(\frac{1}{{\mu_2}^8}\right).
\end{align*}
It is easy to check that for ${\mu_2}\geq3$,
\begin{align}\label{3.10}
d_{\mu_2}<{\mu_2^-}<w_{\mu_2},
\quad
\overline d_{\mu_2}<{\mu_2^+}<\overline w_{\mu_2},
\end{align}
where
\begin{align}\label{3.12}
\left\{ \begin{lgathered}
d_{\mu_2}={\mu_2}-\frac{\pi^2}{4{\mu_2}}-
\frac{\pi^4}{32{\mu_2}^3}-\frac{\pi^6}{128{\mu_2}^5}-
\frac{5\pi^8}{128\times8{\mu_2}^7},
\\[3pt]
w_{\mu_2}={\mu_2}-\frac{\pi^2}{4{\mu_2}}-
\frac{\pi^4}{32{\mu_2}^3}-\frac{\pi^6}{128{\mu_2}^5},
\\[3pt]
\overline d_{\mu_2}={\mu_2}+\frac{\pi^2}{4{\mu_2}}-
\frac{\pi^4}{32{\mu_2}^3}+\frac{\pi^6}{128{\mu_2}^5}-
\frac{5\pi^8}{128\times8{\mu_2}^7},
\\[3pt]
\overline w_{\mu_2}={\mu_2}+\frac{\pi^2}{4{\mu_2}}-
\frac{\pi^4}{32{\mu_2}^3}+\frac{\pi^6}{128{\mu_2}^5}.
\end{lgathered}\right.
\end{align}
Applying \eqref{3.12} into \eqref{3.10}, we find when ${\mu_2}\geq3$,
\[-\frac{\pi^4}{16{\mu_2}^3}-\frac{5\pi^8}{128\times4{\mu_2}^7}
<{\mu_2^-}+{\mu_2^+}-2{\mu_2}<-\frac{\pi^4}{16{\mu_2}^3}.\]
It follows that,
\begin{align*}
\exp{\left(-\frac{\pi^4}{16{\mu_2}^3}
-\frac{5\pi^8}{128\times4{\mu_2}^7}\right)}<
\exp{({\mu_2^-}+{\mu_2^+}-2{\mu_2})}
&<\exp{\left(-\frac{\pi^4}{16{\mu_2}^3}\right)}.
\end{align*}
Note that for $s<0$,$$1+s<e^s<1+s+s^2,$$
here we come that
\begin{align}\label{3.15}
1-\frac{\pi^4}{16{\mu_2}^3}-\frac{5\pi^8}{128\times4{\mu_2}^7}
<\exp{({\mu_2^-}+{\mu_2^+}-2{\mu_2})}
&<1-\frac{\pi^4}{16{\mu_2}^3}+\frac{\pi^8}{256{\mu_2}^6}.
\end{align}
Next we estimate $U'_2(n)$ and $V'_2(n)$. Let
\begin{align}\label{3.17}
P_l(n)&=\frac{1}{{\mu_2^-}^6{\mu_2^+}^6}\left({\mu_2^-}^6
-\frac{3}{8}{\mu_2^-}^4w_{\mu_2}-\frac{15}{128}{\mu_2^-}^4
-\frac{105}{1024}{\mu_2^-}^2w_{\mu_2}\right.
\nonumber\\
&\qquad \qquad\qquad \left.-\frac{4725}{32768}{\mu_2^-}^2-\frac{72765}{262144}w_{\mu_2}-31\right) \nonumber\\
&\qquad\qquad\times \left({\mu_2^+}^6-\frac{3}{8}{\mu_2^+}^4\overline w_{\mu_2}-\frac{15}{128}{\mu_2^+}^4-\frac{105}{1024}{\mu_2^+}^2\overline w_{\mu_2}\right.
\nonumber\\
&\left.\qquad\qquad\qquad
-\frac{4725}{32768}{\mu_2^+}^2-\frac{72765}{262144}\overline w_{\mu_2}-31\right)
\end{align}
and
\begin{align}\label{3.18}
P_r(n)&=\frac{1}{{\mu_2^-}^6{\mu_2^+}^6}\left({\mu_2^-}^6
-\frac{3}{8}{\mu_2^-}^4d_{\mu_2}-\frac{15}{128}{\mu_2^-}^4
-\frac{105}{1024}{\mu_2^-}^2d_{\mu_2}
\right.
\nonumber\\[5pt]
&\left.\qquad\qquad\qquad-\frac{4725}{32768}{\mu_2^-}^2
-\frac{72765}{262144}d_{\mu_2}+31\right)
\nonumber\\[5pt]
&\qquad \qquad\times\left({\mu_2^+}^6-\frac{3}{8}{\mu_2^+}^4\overline d_{\mu_2}-\frac{15}{128}{\mu_2^+}^4-\frac{105}{1024}{\mu_2^+}^2\overline d_{\mu_2}\right.
\nonumber\\[5pt]
&\left.\qquad\qquad\qquad-\frac{4725}{32768}{\mu_2^+}^2
-\frac{72765}{262144}\overline d_{\mu_2}+31\right).
\end{align}
Applying \eqref{3.10} into \eqref{3.8} and \eqref{3.9}, we have for ${\mu_2}\geq3$,
\begin{align}\label{3.19}
U'_2(n)\geq\frac{P_l(n)}{\left(E_I({\mu_2})+\frac{31}{{\mu_2}^6}\right)^2},
\quad
V'_2(n)\leq\frac{P_r(n)}{\left(E_I({\mu_2})-\frac{31}{{\mu_2}^6}\right)^2}.
\end{align}
To bound $U'_2(n)$ and $V'_2(n)$ in terms of ${\mu_2}$, we shall compute that for ${\mu_2}\geq167$,
\begin{align}\label{3.222}
&\frac{P_l(n)}{\left(E_I({\mu_2})+\frac{31}{{\mu_2}^6}\right)^2}
-\left(1-\frac{7}{{\mu_2}^5}-\frac{115}{{\mu_2}^6}\right)\geq0,
\end{align}
and for ${\mu_2}\geq31$
\begin{align}\label{3.233}
\frac{P_r(n)}{\left(E_I({\mu_2})-\frac{31}{{\mu_2}^6}\right)^2}
-\left(1-\frac{6}{{\mu_2}^5}+\frac{115}{{\mu_2}^6}\right)\leq0.
\end{align}
Rewrite \eqref{3.222} and \eqref{3.233} as
\begin{equation*}
{\mu_2}^{12}P_l(n)-\left(1-\frac{7}{{\mu_2}^5}
-\frac{115}{{\mu_2}^6}\right)\left({\mu_2}^6E_I({\mu_2})+31\right)^2\geq0
\end{equation*}
and
\begin{equation*}
{\mu_2}^{12}P_r(n)-\left(1-\frac{6}{{\mu_2}^5}
+\frac{115}{{\mu_2}^6}\right)\left({\mu_2}^6E_I({\mu_2})-31\right)^2\leq0.
\end{equation*}
Substituting \eqref{3.12} into \eqref{3.17} and \eqref{3.18}, respectively, we can therefore compute that
\begin{align*}
&{\mu_2}^{12}P_l(n)-\left(1-\frac{7}{{\mu_2}^5}
-\frac{115}{{\mu_2}^6}\right)\left({\mu_2}^6E_I({\mu_2})+31\right)^2
=\frac{\sum_{i=0}^{25}a_i{\mu_2}^i}{17592186044416{\mu_2}^6
\left(4{\mu_2}^4-\pi^4\right)^3},
\\[5pt]
&{\mu_2}^{12}P_r(n)-\left(1-\frac{6}{{\mu_2}^5}
+\frac{115}{{\mu_2}^6}\right)\left({\mu_2}^6E_I({\mu_2})-31\right)^2
=-\frac{\sum_{i=0}^{25}b_i{\mu_2}^i}{1125899906842624{\mu_2}^6
\left(4{\mu_2}^4-\pi^4\right)^3},
\end{align*}
where $a_i$ and $b_i$ are real numbers, here we just list $a_{23}$-$a_{25}$ and $b_{23}$-$b_{25}$:
\begin{align*}
a_{23}&=-45493393110859776-86586540687360\pi^4,\\ a_{24}&=-16044073672507392-46179488366592\pi^4,\\
a_{25}&=7881299347898368-79164837199872 \pi^4,\\
b_{23}&=-2823756966361300992 +5541538603991040  \pi^4,\\
b_{24}&=-324259173170675712+2955487255461888  \pi^4,\\
b_{25}&=-432345564227567616+5066549580791808\pi^4.
\end{align*}
It can be checked that for any $0\leq i\leq22$ and ${\mu_2}\geq 24$,
\[-|a_i|{\mu_2}^i\geq-|a_{23}|{\mu_2}^{23},\quad
-|b_i|{\mu_2}^i\geq-|b_{23}|{\mu_2}^{23}.\]
So we could obtain that for ${\mu_2}\geq 24$,
\begin{align*}
\sum_{i=0}^{25}a_i{\mu_2}^i
&\geq -\sum_{i=0}^{23}|a_i|{\mu_2}^i+a_{24}{\mu_2}^{24}
+a_{25}{\mu_2}^{25}
\geq-24|a_{23}|{\mu_2}^{23}+a_{24}{\mu_2}^{24}
+a_{25}{\mu_2}^{25}
\end{align*}
and
\begin{align*}
\sum_{i=0}^{25}b_i{\mu_2}^i&\geq -\sum_{i=0}^{23}|b_i|{\mu_2}^i+b_{24}{\mu_2}^{24}+b_{25}{\mu_2}^{25}
\geq -24|b_{23}|{\mu_2}^{23}+b_{24}{\mu_2}^{24}+b_{25}{\mu_2}^{25}.
\end{align*}
Meanwhile, we can calculate that for ${\mu_2}\geq167 $,
\[(-24|a_{23}|+a_{24}{\mu_2}+a_{25}{\mu_2}^{2}){\mu_2}^{23}\geq0,\]
and for ${\mu_2}\geq31,$
\[(-24|b_{23}|+b_{24}{\mu_2}+b_{25}{\mu_2}^{2}){\mu_2}^{23}\geq0.\]
So \eqref{3.222} and \eqref{3.233} are established.
Combining them with \eqref{3.19},
we attain for ${\mu_2}\geq167$,
\begin{align}\label{3.25}
U'_2(n)\geq 1-\frac{7}{{\mu_2}^5}-\frac{115}{{\mu_2}^6},
\quad
V'_2(n)\leq1-\frac{6}{{\mu_2}^5}+\frac{115}{{\mu_2}^6}.
\end{align}
Applying \eqref{3.15} and \eqref{3.25} into \eqref{3.6}, we establish the validity of the Lemma \ref{lem3.3} when $k=2$.
The remaining cases follow a similar pattern to the case of $k=2$, and their proofs are omitted here. Thus, the proof is completed.
\end{proof}

In order to provide the asymptotic formula for $\p_k(n)$, we begin by reviewing some definitions and notations introduced by Chern \cite{chern-eta}.
Let us define
\[G(q)=G(e^{2\pi \mathnormal{i} \tau}):=\prod_{r=1}^R{(q^{m_r};q^{m_r})}_\infty ^{\delta_r},\]
where $\mathbf{m}=(m_1,\dots,m_R)$ and $ \mathbf{\delta}=(\delta_1,\dots,\delta_R)$.
Assuming that $k$ and $h$ are positive integers with $\text{gcd}(h,k)=1$, we define
\begin{align*}
&\Delta_1=-\frac{1}{2}\sum_{r=1}^R \delta_r,\qquad \Delta_2=\sum_{r=1}^R m_r\delta_r,
\\[5pt]
&\Delta_3(k)=-\sum_{r=1}^R\frac{\delta_r\mathrm{gcd}^2(m_r,k)}{m_r},\qquad\  \Delta_4(k)=\prod_{r=1}^R\left(\frac{m_r}{\mathrm{gcd}(m_r,k)}\right)
^{-\frac{\delta_r}{2}},
\\[5pt]
&\hat A_k(n)=\mathop\sum_{0\leq h<k \atop\ \mathrm{gcd}(h,k)=1}\exp{\left(-\frac{2\pi nhi}{k}-\pi i\sum_{r=1}^R\delta_rs\left(\frac{m_rh}{\mathrm{gcd}(m_r,k)},
\frac{k}{\mathrm{gcd}(m_r,k)}\right)\right)},
\end{align*}
where $s(h,j)$ is the Dedekind sum defined by$$s(h,j)=\sum_{r=1}^{j-1}\left(\frac{r}{j}-\left[\frac{r}{j}\right]-\frac{1}{2}\right)\left(\frac{hr}{j}-\left[\frac{hr}{j}\right]-\frac{1}{2}\right).$$
Let $L=lcm(m_1,\dots,m_R)$, we divide the set $\{1,2,\dots,L\}$ into two disjoint subsets:$$\mathcal{L}_{>0}:=\{1\leq l \leq L:\Delta_3(l)>0\},\qquad\mathcal{L}_{\leq0}:=\{1\leq l \leq L:\Delta_3(l)\leq0\}.$$
If we write $$G(q)=\sum_{n=0}^\infty g(n)q^n.$$
Chern \cite[Theorem~1.1]{chern-eta} obtained the following asymptotic formula for $g(n)$ with $\Delta_1\leq0$.
\begin{lem}\label{thm3.1}
If $\Delta_1\leq0$ and the inequality
\begin{equation}\label{3.29}
\min_{1\leq r\leq R}\left(\frac{\mathrm{gcd}^2(m_r,l)}{m_r}\right)\geq \frac{\Delta_3(l)}{24}
\end{equation}
holds for all $1\leq l\leq L$, then for positive integers  $n>-\frac{\Delta_2}{24}$, we have
\begin{align*}
g(n)=E(n)+&\sum_{l\in \mathcal{L}_{>0}}2\pi\Delta_4(l)\left(\frac{24n+\Delta_2}{\Delta_3(l)}\right)
^{-\frac{\Delta_1+1}{2}}
\sum_{1\leq k<N \atop\ k\equiv L^l}\frac{I_{-\Delta_1-1}(\frac{\pi}{6k}\sqrt{\Delta_3(l)(24n+\Delta_2)})}{k}\hat A_k(n),
\end{align*}
where $\zeta(\cdot)$ is the Riemann zeta-function, $I_v(s)$ is the $v$-th modified Bessel function of the first kind,
\begin{align*}
\Xi_{\Delta_1}(t)
:=\begin{cases}1,&\Delta_1=0,
\\
2\sqrt t,&\Delta_1=-\frac{1}{2},
\\
t\log{(t+1)},&\Delta_1=-1,
\\
\zeta(-\Delta_1)t^{-2\Delta_1-1},&\mathrm{otherwise},
\\ \end{cases}
\end{align*}
and
\begin{align*}
|E(n)|\leq &\frac{2^{-\Delta_1}\pi^{-1}N^{-\Delta_1+2}}
{n+\frac{\Delta_2}{24}}\exp\left(2\pi\left(n+\frac{\Delta_2}{24}\right)
N^{-2}\right)\sum_{l\in \mathcal{L}_{>0}}\Delta_4(l)\exp\left(\frac{\Delta_3(l)\pi}{3}\right)
\\[5pt]
&+2\exp\left(2\pi\left(n+\frac{\Delta_2}{24}\right)
N^{-2}\right)\Xi_{\Delta_1}(N)
\\[5pt]
&\times \left(\sum_{1\leq l \leq L}\Delta_4(l)\exp\left(\frac{\pi\Delta_3(l)}{24}
+\sum_{r=1}^{R}\frac{|\delta_r|\exp(-\pi\mathrm{gcd}^2(m_r,l)/m_r)}
{(1-\exp(-\pi\mathrm{gcd}^2(m_r,l)/m_r))^2}\right)\right.
\\[5pt]
&\left. \qquad -\sum_{l\in \mathcal{L}_{>0}}\Delta_4(l)\exp\left(\frac{\pi\Delta_3(l)}{24}\right)\right). \end{align*}
\end{lem}

We will now use Lemma \ref{thm3.1} to prove Theorem \ref{thm2}.

{\noindent{\it Proof of Theorem \ref{thm2}}}. For $2\leq k\leq9$, the proof process is similar. Therefore, here we only demonstrate the case of $k=2$.
Let $k=2$ in \eqref{gene}, and based on the notation defined by Chern, we have $\mathbf{m}=(1,2,2,4)$ and $\delta=(-2,1,2,-1)$.
The calculations show that $\Delta_1=0$ and $\Delta_2=0$.
Furthermore, we find that $L=4$. The values of $\Delta_3(l)$ and $\Delta_4(l)$ for $1\leq l\leq L$ are presented in following table.
\begin{table}[h]
\centering
\renewcommand{\arraystretch}{1.5}
  \begin{tabular}{p{2cm}|p{2cm}|p{2cm}|p{2cm}|p{2cm}}
  \centering$l$           & \centering$1$                & \centering$2$      & \centering$3$                &~~~~~~~$4$
  \\[5pt] \hline \rule{0pt}{20pt}
  \centering$\Delta_3(l)$ &\centering {\large$\frac{3}{4}$}      &\centering $-3$     & \centering {\large$\frac{3}{4}$}      &~~~~~~~$0$
  \\[5pt] \hline \rule{0pt}{20pt}
  \centering$\Delta_4(l)$ &\centering {\large$\frac{\sqrt2}{2}$} & \centering$\sqrt2$ & \centering{\large$\frac{\sqrt2}{2}$} &~~~~~~~$1$
  \\[5pt]
  \end{tabular}
  \caption{The values of $\Delta_3(l)$ and $\Delta_4(l)$ for $1\leq l\leq4$.}
  \label{delta}
\end{table}
As a result, we have $\mathcal{L}_{>0}=\{1,3\}$.

Since in this case \eqref{3.29} is always satisfied, we can apply Lemma \ref{thm3.1}, which yields the following result:
\[\overline p_2(n)=E_2(n)+\frac{\pi^2}{\sqrt8{\mu_2}}\mathop\sum_{1\leq k\leq N \atop\ 2\nmid k}I_1\left(\frac{{\mu_2}}{k}\right)\frac{\hat A_k(n)}{k},\]
where
\begin{align*}
|E_2(n)|&\leq \frac{\pi^{-1}N^2}{n}\exp{\left( 2\pi nN^{-2}\right)}\cdot 2\frac{\sqrt2}{2}\exp\left(\frac{\pi}{4}\right)+2\exp\left(2\pi nN^{-2}\right)
\Bigg\{\sum_{1\leq l\leq4}\Delta_4(l)
\\
&\quad\times \exp\left(\frac{\pi\Delta_3(l)}{24}+\sum_{r=1}^{4}
\frac{|\delta_r|\exp(-\pi\mathrm{gcd}^2(m_r,l)/m_r)}
{(1-\exp(-\pi\mathrm{gcd}^2(m_r,l)/m_r))^2}\right)
-\sqrt2\exp\left(\frac{\pi}{32}\right)\Bigg\}.
\end{align*}

Let's assume $N=\lfloor {\mu_2}\rfloor$. Then, we have the following:
\begin{align}\label{3.299}
|E_2(n)|&\leq \frac{\sqrt2\pi}{2}\frac{{\lfloor {\mu_2}\rfloor}^2}{{\mu_2}^2}\exp\left( \frac{4{\mu_2}^2}{\pi {\lfloor {\mu_2}\rfloor}^2}+\frac{\pi}{4}\right)+2\exp\left(\frac{4{\mu_2}^2}{\pi {\lfloor {\mu_2}\rfloor}^2}\right)
\nonumber\\[5pt]
&\quad\times \left\{\sum_{1\leq l\leq4}\Delta_4(l)\exp\left(\frac{\pi\Delta_3(l)}{24}
+\sum_{r=1}^{4}\frac{|\delta_r|\exp(-\pi\mathrm{gcd}^2(m_r,l)/m_r)}
{(1-\exp(-\pi\mathrm{gcd}^2(m_r,l)/m_r))^2}\right)\right.
\nonumber\\[5pt]
&\quad\qquad-\sqrt2\exp\left(\frac{\pi}{32}\right)\Bigg\},
\end{align}
and using the following inequalities for ${\mu_2}\geq4$,
\[\frac{\lfloor {\mu_2}\rfloor^2}{{\mu_2}^2}\leq1, \quad \frac{{\mu_2}^2}{\lfloor {\mu_2}\rfloor^2}<\frac{{\mu_2}^2}{({\mu_2}-1)^2}<2,\]
connecting \eqref{3.299} with values in Table \ref{delta}, it follows
$|E_2(n)|\leq 631.$
Observing $\hat A_1(n)=1$, we can further conclude that
\[\overline p_2(n)=\frac{\pi^2}{\sqrt8{\mu_2}}I_1({\mu_2})+R_2(n),\]
where
\[R_2(n)=E_2(n)+\frac{\pi^2}{\sqrt8{\mu_2}}\mathop\sum_{3\leq k\leq N \atop\ 2\nmid k}I_1\left(\frac{{\mu_2}}{k}\right)\frac{\hat A_k(n)}{k}.\]
Next, we come to show that for ${\mu_2}\geq22,$
\begin{equation}\label{3.30}|R_2(n)|\leq \frac{\sqrt3\pi^{\frac{3}{2}}}{2{\mu_2}^{\frac{1}{2}}}\exp\left(\frac{{\mu_2}}{3}\right).\end{equation}
Since for any $n\geq0$, $k\geq1$, $|\hat A_k(n)|\leq k$ and $|e^{2\pi si}|=1$ for any $s\in \mathbb{R}$, we obtain that
\begin{align*}\left|
\frac{\pi^2}{\sqrt8{\mu_2}}\mathop\sum_{3\leq k\leq N \atop\ 2\nmid k}I_1\left(\frac{{\mu_2}}{k}\right)\frac{\hat A_k(n)}{k}\right|
\leq
\frac{\pi^2}{\sqrt8{\mu_2}}\mathop\sum_{3\leq k\leq N \atop\ 2\nmid k}I_1\left(\frac{{\mu_2}}{k}\right)
\leq \frac{\pi^2}{\sqrt8{\mu_2}}\frac{\lfloor {\mu_2}\rfloor}{2}I_1\left(\frac{{\mu_2}}{3}\right)
\leq \frac{\pi^2}{2\sqrt8}I_1\left(\frac{{\mu_2}}{3}\right).
\end{align*}
Recall that Bringmann, Kane, Rolen and Trippin \cite{bringmann-i-bessel} stated that for $s\geq1$,
\begin{align*}
I_1(s)\leq\sqrt{\frac{2}{\pi s}}e^s,
\end{align*}
which yields that
\begin{align*}\left|
\frac{\pi^2}{\sqrt8{\mu_2}}\mathop\sum_{3\leq k\leq N \atop\ 2\nmid k}I_1\left(\frac{{\mu_2}}{k}\right)\frac{\hat A_k(n)}{k}\right|
\leq \frac{\pi^2}{2\sqrt8}\sqrt{\frac{2}{\pi \frac{{\mu_2}}{3}}}\exp{\left(\frac{{\mu_2}}{3}\right)}
\leq \frac{\sqrt3\pi^{\frac{3}{2}}}{4{\mu_2}^{\frac{1}{2}}}\exp{\left(\frac{{\mu_2}}{3}\right)}.
\end{align*}
Thus, we get for ${\mu_2}\geq4$,
\begin{align*}
|R_2(n)|\leq 631+\frac{\sqrt3\pi^{\frac{3}{2}}}{4{\mu_2}^{\frac{1}{2}}}
\exp{\left(\frac{{\mu_2}}{3}\right)}.
\end{align*}
Moreover, for ${\mu_2}\geq22$, it can be easily checked that,
\[631<\frac{\sqrt3\pi^{\frac{3}{2}}}{4{\mu_2}^{\frac{1}{2}}}
\exp{\left(\frac{{\mu_2}}{3}\right)}.\]
So \eqref{3.30} is valid, which completes the proof.
\qed

In accordance with Theorem \ref{thm2}, we are able to derive the upper and lower bounds for $\p_k(n)$ when $2\leq k\leq9$.
\begin{cor}\label{thm3}
For $2\leq k\leq9$ and $\mu_k\geq \dot{n}_k$, we have
\begin{align}\label{up-low}
{M_k(n)}\left(1-\frac{1}{{\mu_k}^6}\right)\leq \p_k(n)\leq {M_k(n)}\left(1+\frac{1}{{\mu_k}^6}\right),
\end{align}
where
\begin{align*}
&\dot{n}_2=43,\quad \dot{n}_3=49,\quad \dot{n}_4=43,\quad \dot{n}_5=58,\\
&\dot{n}_6=130,\quad \dot{n}_7=102,\quad \dot{n}_8=129,\quad \dot{n}_9=268.
\end{align*}
Here, $M_k(n)=C_k(n)I_1(\mu_k)$, and  $C_k(n)$ is defined in Theorem \ref{thm2}.
\end{cor}

\begin{proof}
Let us define
\begin{equation*}
G_k(n):=\frac{R_k'(n)}{C_k(n)I_1({\mu_k})},
\end{equation*}
where $R_k'(n)$ is defined in Theorem \ref{thm2}.

By utilizing Theorem \ref{thm2}, we can establish the following inequalities for $2\leq k\leq9$ when $\mu_k\geq n_k$:
\[{M_k(n)}(1-G_k(n))\leq \overline p_k(n) \leq {M_k(n)}(1+G_k(n)).\]
First, we will prove the case of $k=2$. The expression for $G_2$ can be simplified as follows.
\begin{equation*}
G_2(n):=
\sqrt{\frac{6{\mu_2}}{\pi}}
\frac{\exp{\left(\frac{{\mu_2}}{3}\right)}}{I_1({\mu_2})}.
\end{equation*}
According to the proof of Theorem 1.3 in \cite{dong-ji}, we can demonstrate that for $x\geq43$,
\[G(x):=\sqrt{\frac{6{x}}{\pi}}
\frac{\exp{\left(\frac{{x}}{3}\right)}}{I_1({x})}\leq \frac{1}{{x}^6},\]
which in turn implies that for $\mu_2\geq43$,
\[G_2(n)\leq \frac{1}{{\mu_2}^6}.\]
Thus, we have verified the validity of \eqref{up-low} when $k=2$.
The proofs for other cases are similar, and we omit the details here.
This completes the proof.
\end{proof}

\section{Proof of Theorem \ref{thm1.5} }\label{Sec-4}
In this section, our goal is to demonstrate the log-concavity and third-order Tur\'{a}n inequalities of $\p_k(n)$ for $2\leq k\leq9$. To begin, we will establish both upper and lower bounds for $Q_k(n)$, where
\begin{align*}
Q_k(n)=\frac{\p_k(n-1)\p_k(n+1)}{\p_k(n)^2}.
\end{align*}

\begin{thm}\label{thm4}
For $\mu_k\geq\ddot{n}_k$, that is $n\geq \tilde{n}_k$, we establish the following inequalities:
\begin{align}\label{Qkn-bounds}
\tilde{L}_k(n)<Q_k(n)<\tilde{R}_k(n),
\end{align}
where all these values are defined in following table.

\begin{table}[h]
\begin{tabular}{c|c|c|c|c}
  $k$ & $\ddot{n}_k$ & $\tilde{n}_k$ & $\tilde{L}_k(n)$ & $\tilde{R}_k(n)$
  \\[3pt]   \hline \rule{0pt}{23pt}
  $2$ & $167$ & $5652$
  & $1$ {\large$-\frac{\pi^4}{16{\mu_2}^3}+\frac{3\pi^4}{16{\mu_2}^4}
  -\frac{7}{{\mu_2}^5}-\frac{130}{{\mu_2}^6}$}
  & $1${\large $-\frac{\pi^4}{16{\mu_2}^3}+\frac{3\pi^4}{16{\mu_2}^4}
  -\frac{6}{{\mu_2}^5}+\frac{120+\frac{\pi^8}{256}}{{\mu_2}^6}$}
  \\[9pt] \hline \rule{0pt}{23pt}
  $3$ & $49$ & $365$
  & $1$ {\large$-\frac{\pi^4}{9\mu_3^3}+\frac{\pi^4}{3\mu_3^4}
  -\frac{13}{\mu_3^5}-\frac{200}{\mu_3^6}$}
  & $1$ {\large$-\frac{\pi^4}{9\mu_3^3}+\frac{\pi^4}{3\mu_3^4}
  -\frac{6}{\mu_3^5}+\frac{146+\frac{\pi^8}{81}}{\mu_3^6}$}
  \\[9pt] \hline \rule{0pt}{23pt}
  $4$ & $58$ & $455$
  & $1$ {\large$-\frac{9\pi^4}{64\mu_4^3}+\frac{27\pi^4}{64\mu_4^4}
  -\frac{16}{\mu_4^5}-\frac{300}{\mu_4^6}$}
  & $1$ {\large$-\frac{9\pi^4}{64\mu_4^3}+\frac{27\pi^4}{64\mu_4^4}
  -\frac{15}{\mu_4^5}+\frac{150+\frac{81\pi^8}{4096}}{\mu_4^6}$}
  \\[9pt] \hline \rule{0pt}{23pt}
  $5$ & $94$ & $1120$
  & $1$ {\large$-\frac{4\pi^4}{25\mu_5^3}+\frac{12\pi^4}{25\mu_5^4}
  -\frac{18}{\mu_5^5}-\frac{400}{\mu_5^6}$}
  & $1$ {\large$-\frac{4\pi^4}{25\mu_5^3}+\frac{12\pi^4}{25\mu_5^4}
  -\frac{17}{\mu_5^5}+\frac{400}{\mu_5^6}$}
  \\[9pt] \hline \rule{0pt}{23pt}
  $6$ & $130$ & $2055$
  & $1$ {\large$-\frac{25\pi^4}{144\mu_6^3}+\frac{25\pi^4}{48\mu_6^4}
  -\frac{20}{\mu_6^5}-\frac{441}{\mu_6^6}$}
  & $1$ {\large$-\frac{25\pi^4}{144\mu_6^3}+\frac{25\pi^4}{48\mu_6^4}
  -\frac{19}{\mu_6^5}+\frac{441}{\mu_6^6}$}
  \\[9pt] \hline \rule{0pt}{23pt}
  $7$ & $102$ & $1230$
  & $1$ {\large$-\frac{9\pi^4}{49\mu_7^3}+\frac{27\pi^4}{49\mu_7^4}
  -\frac{21}{\mu_7^5}-\frac{500}{\mu_7^6}$}
  & $1$ {\large$-\frac{9\pi^4}{49\mu_7^3}+\frac{27\pi^4}{49\mu_7^4}
  -\frac{20}{\mu_7^5}+\frac{500}{\mu_7^6}$}
  \\[9pt] \hline \rule{0pt}{23pt}
  $8$ & $300$ & $10422$
  & $1$ {\large$-\frac{49\pi^4}{256\mu_8^3}+\frac{147\pi^4}{256\mu_8^4}
  -\frac{21}{\mu_8^5}-\frac{505}{\mu_8^6}$}
  & $1$ {\large$-\frac{49\pi^4}{256\mu_8^3}+\frac{147\pi^4}{256\mu_8^4}
  -\frac{20}{\mu_8^5}+\frac{505}{\mu_8^6}$}
  \\[9pt] \hline \rule{0pt}{23pt}
  $9$ & $268$ & $8187$
  & $1$ {\large$-\frac{16\pi^4}{81\mu_9^3}+\frac{16\pi^4}{27\mu_9^4}
  -\frac{22}{\mu_9^5}-\frac{524}{\mu_9^6}$}
  & $1$ {\large$-\frac{16\pi^4}{81\mu_9^3}+\frac{16\pi^4}{27\mu_9^4}
  -\frac{21}{\mu_9^5}+\frac{529}{\mu_9^6}$ }
  \\[3pt]
\end{tabular}
\caption{Values of $\tilde{n}_k$, $\tilde{L}_k(n)$ and $\tilde{R}_k(n)$ for $2\leq k\leq9$.}
\label{Qk(n)-bounds}
\end{table}
\end{thm}

\vspace{3cm}
\begin{proof}
Define
\begin{align*}
A_k(n)&=\frac{M_k(n-1)M_k(n+1)}{{M_k(n)}^2},
\end{align*}
where $M_k(n)$ is defined in Corollary \ref{thm3}.

From Corollary \ref{thm3}, we see that for $\mu_k\geq \tilde{n}_k$,
\begin{equation}\label{3.344}
A_k(n)L_{Q_k(n)}\leq {Q_k(n)}\leq A_k(n)R_{Q_k(n)},
\end{equation}
where
\begin{equation}\label{3.35}
L_{Q_k(n)}=\frac{\left(1-\frac{1}{{{\mu_k^-}}^6}\right)
\left(1-\frac{1}{{{\mu_k^+}}^6}\right)}{\left(1+\frac{1}{{{\mu_k}}^6}\right)^2},
\quad
R_{Q_k(n)}=\frac{\left(1+\frac{1}{{{\mu_k^-}}^6}\right)
\left(1+\frac{1}{{{\mu_k^+}}^6}\right)}{\left(1-\frac{1}{{{\mu_k}}^6}\right)^2}.
\end{equation}
According to Theorem \ref{thm2}, $A_k(n)$ can be simplified as
\begin{align}\label{Akn-1}
A_k(n)&=\frac{{\mu_k}^2I_1({{\mu_k^-}})
I_1({{\mu_k^+}})}{{{\mu_k^-}}{{\mu_k^+}}I_1({\mu_k})^2}.
\end{align}
By Lemma \ref{lem3.3}, the upper and lower bounds of \eqref{Akn-1} are given by
\begin{align}\label{Akn-bounds}
\frac{{{\mu_k}}^3}{\sqrt{{{\mu_k^-}}^3{{\mu_k^+}}^3}}U_k(n)\leq A_k(n)\leq\frac{{{\mu_k}}^3}{\sqrt{{{\mu_k^-}}^3{{\mu_k^+}}^3}}V_k(n).
\end{align}
Next, we aim to provide upper and lower bounds for
\[\frac{{{\mu_k}}^3}{\sqrt{{{\mu_k^-}}^3{{\mu_k^+}}^3}}.\]
We begin by considering the case of $k=2$.
It can be established that for ${{\mu_2}}\geq3$,
\begin{equation}\label{3.355}
1+\frac{3\pi^4}{16{{\mu_2}}^4}+\frac{\pi^8}{35{{\mu_2}}^8}
\leq\frac{{{\mu_2}}^3}{\sqrt{{{\mu_2^-}}^3{{\mu_2^+}}^3}}\leq1
+\frac{3\pi^4}{16{{\mu_2}}^4}+\frac{\pi^8}{12{{\mu_2}}^8},
\end{equation}
which is equivalent to
\begin{align}\label{3.366}
{{\mu_2}}^{12}-{{\mu_2^-}}^6{{\mu_2^+}}^6
{\left(1+\frac{3\pi^4}{16{{\mu_2}}^4}+\frac{\pi^8}{35{{\mu_2}}^8}\right)}^4
&\geq0,
\\[5pt] \label{3.377}
{{\mu_2}}^{12}-{{\mu_2^-}}^6{{\mu_2^+}}^6
{\left(1+\frac{3\pi^4}{16{{\mu_2}}^4}+\frac{\pi^8}{12{{\mu_2}}^8}\right)}^4
&\leq0.
\end{align}
By utilizing \eqref{3.100}, we can express \eqref{3.366} and \eqref{3.377} in the following form:
\begin{align*}
&{{\mu_2}}^{12}-{{\mu_2^-}}^6{{\mu_2^+}}^6
\left(1+\frac{3\pi^4}{16{{\mu_2}}^4}+\frac{\pi^8}{35{{\mu_2}}^8}\right)^4
\\[5pt]
&=\frac{1}{6294077440000{{\mu_2}}^{32}}
\left(313298944000 \pi^8 {{\mu_2}}^{36}+177899008000 \pi^{12} {{\mu_2}}^{32}\right.
\\[5pt]
&\left.\qquad +3408753600 \pi^{16} {{\mu_2}}^{28}-5446075600 \pi^{20} {{\mu_2}}^{24}-1992890060 \pi^{24} {{\mu_2}}^{20}\right.
\\[5pt]
&\qquad -167858335 \pi^{28} {{\mu_2}}^{16}+19789376 \pi^{32} {{\mu_2}}^{12}+8611328 \pi^{36} {{\mu_2}}^8
\\[5pt]
&\left.\qquad+933888 \pi^{40} {{\mu_2}}^4+65536 \pi^{44}\right)
\end{align*}and
\begin{align*}
&{{\mu_2}}^{12}-{\mu_2^-}^6{\mu_2^+}^6
{\left(1+\frac{3\pi^4}{16{{\mu_2}}^4}+\frac{\pi^8}{12{{\mu_2}}^8}\right)}^4\\
&=-\frac{1}{339738624{{\mu_2}}^{32}}\left(57507840 \pi^8 {{\mu_2}}^{36}-23556096 \pi^{12} {{\mu_2}}^{32}+2714688 \pi^{16} {{\mu_2}}^{28}\right.\\
&\left.\qquad-3100464 \pi^{20} {\mu_2}^{24}+580044 \pi^{24} {{\mu_2}}^{20}-94689 \pi^{28}{{\mu_2}}^{16}+51952 \pi^{32} {{\mu_2}}^{12}\right.\\
&\left.\qquad-4704 \pi^{36}{{\mu_2}}^8+768 \pi^{40} {{\mu_2}}^4-256 \pi^{44}\right).
\end{align*}
We can further compute that when ${{\mu_2}}\geq3$,
\begin{align*}
\qquad&313298944000 \pi^8 {{\mu_2}}^{36}+177899008000 \pi^{12} {{\mu_2}}^{32}+3408753600 \pi^{16} {{\mu_2}}^{28}\\
& \quad-5446075600 \pi^{20} {{\mu_2}}^{24}-1992890060 \pi^{24} {{\mu_2}}^{20}-167858335 \pi^{28} {{\mu_2}}^{16}\\
&\quad +19789376 \pi^{32} {{\mu_2}}^{12}+8611328 \pi^{36} {{\mu_2}}^8+933888 \pi^{40} {{\mu_2}}^4+65536 \pi^{44}\geq0
\end{align*}
and when $\mu_2\geq4$,
\begin{align*}
&57507840 \pi^8 {{\mu_2}}^{36}-23556096 \pi^{12} {{\mu_2}}^{32}+2714688 \pi^{16} {{\mu_2}}^{28}\\
&\quad-3100464 \pi^{20} {\mu_2}^{24}+580044 \pi^{24} {{\mu_2}}^{20}-94689 \pi^{28}{{\mu_2}}^{16}+51952 \pi^{32} {{\mu_2}}^{12}\\
&\quad-4704 \pi^{36}{{\mu_2}}^8+768 \pi^{40} {{\mu_2}}^4-256 \pi^{44}\geq0.
\end{align*}
Therefore, we have verified \eqref{3.366} and \eqref{3.377}, and thus obtained \eqref{3.355}.
Substituting \eqref{3.355} into \eqref{Akn-bounds}, we get that for ${\mu_2}\geq167$, i.e. $n\geq5652$,
\begin{align}\label{3.37}
A_2(n)&\geq \left(1+\frac{3\pi^4}{16{{\mu_2}}^4}+\frac{\pi^8}{35{{\mu_2}}^8}\right)
\left(1-\frac{\pi^4}{16{{\mu_2}}^3}-\frac{5\pi^8}{512{{\mu_2}}^7}\right)
\left(1-\frac{7}{{{\mu_2}}^5}-\frac{115}{{{\mu_2}}^6}\right)
\end{align}
and
\begin{align}\label{3.38}
A_2(n)&\leq \left(1+\frac{3\pi^4}{16{{\mu_2}}^4}+\frac{\pi^8}{12{{\mu_2}}^8}\right)
\left(1-\frac{\pi^4}{16{{\mu_2}}^3}+\frac{\pi^8}{256{{\mu_2}}^6}\right)
\left(1-\frac{6}{{{\mu_2}}^5}+\frac{115}{{{\mu_2}}^6}\right).
\end{align}
Then, we will compute the upper bound for $R_{Q_2(n)}$ and the lower bound for $L_{Q_2(n)}$ in \eqref{3.35}.
Applying \eqref{3.100} into \eqref{3.35}, we find that \begin{equation*}L_{Q_2(n)}=\frac{{{\mu_2}}^{12}\left(\left({{\mu_2}}^2+\frac{\pi^2}{2}\right)^3-1\right)\left(\left({{\mu_2}}^2-\frac{\pi^2}{2}\right)^3-1\right)}{({{\mu_2}}^6+1)^2\left({{\mu_2}}^4-\frac{\pi^4}{4}\right)^3}\end{equation*}
and \begin{equation*}R_{Q_2(n)}=\frac{{{\mu_2}}^{12}\left(\left({{\mu_2}}^2+\frac{\pi^2}{2}\right)^3+1\right)\left(\left({{\mu_2}}^2-\frac{\pi^2}{2}\right)^3+1\right)}{({{\mu_2}}^6-1)^2\left({{\mu_2}}^4-\frac{\pi^4}{4}\right)^3}.\end{equation*}Since $L_{Q_2(n)},R_{Q_2(n)}$ are both just polynomial about ${\mu_2}$, it can be readily checked that when ${\mu_2}\geq57$,  \begin{equation}\label{3.39}L_{Q_2(n)}\geq1-\frac{5}{{{\mu_2}}^6}\quad \mathrm{and} \quad R_{Q_2(n)}\leq1+\frac{5}{{{\mu_2}}^6}.\end{equation}
By substituting \eqref{3.37}, \eqref{3.38}, and \eqref{3.39} into \eqref{3.344}, we obtain that for ${{\mu_2}}\geq167$, i.e. $n\geq5652$,
\begin{align}\label{3.40}
{Q_2(n)}
&\geq \left(1+\frac{3\pi^4}{16{{\mu_2}}^4}+\frac{\pi^8}{35{{\mu_2}}^8}\right)
\left(1-\frac{\pi^4}{16{{\mu_2}}^3}-\frac{5\pi^8}{512{{\mu_2}}^7}\right)
\nonumber\\[5pt]
&\quad \times\left(1-\frac{7}{{{\mu_2}}^5}-\frac{115}{{{\mu_2}}^6}\right)
\left(1-\frac{5}{{{\mu_2}}^6}\right)
\end{align}
and
\begin{align}\label{3.41}
{Q_2(n)}
&\leq \left(1+\frac{3\pi^4}{16{{\mu_2}}^4}+\frac{\pi^8}{12{{\mu_2}}^8}\right)
\left(1-\frac{\pi^4}{16{{\mu_2}}^3}+\frac{\pi^8}{256{{\mu_2}}^6}\right)
\nonumber\\[5pt]
&\quad \times\left(1-\frac{6}{{{\mu_2}}^5}+\frac{115}{{{\mu_2}}^6}\right)
\left(1+\frac{5}{{{\mu_2}}^6}\right).
\end{align}
Now we shall  compute  that when ${\mu_2}\geq38$,
\begin{align}\label{3.42}
&\left(1+\frac{3\pi^4}{16{{\mu_2}}^4}+\frac{\pi^8}{35{{\mu_2}}^8}\right)
\left(1-\frac{\pi^4}{16{{\mu_2}}^3}-\frac{5\pi^8}{512{{\mu_2}}^7}\right)
\left(1-\frac{7}{{{\mu_2}}^5}-\frac{115}{{{\mu_2}}^6}\right)
\left(1-\frac{5}{{{\mu_2}}^6}\right)
\nonumber\\[5pt]
&\quad>1-\frac{\pi^4}{16{{\mu_2}}^3}+\frac{3\pi^4}{16{{\mu_2}}^4}
-\frac{7}{{{\mu_2}}^5}-\frac{130}{{{\mu_2}}^6},
\end{align}
and when ${{\mu_2}}\geq16$,
\begin{align}\label{3.43}
&\left(1+\frac{3\pi^4}{16{{\mu_2}}^4}+\frac{\pi^8}{12{{\mu_2}}^8}\right)
\left(1-\frac{\pi^4}{16{{\mu_2}}^3}+\frac{\pi^8}{256{{\mu_2}}^6}\right)
\left(1-\frac{6}{{{\mu_2}}^5}+\frac{115}{{{\mu_2}}^6}\right)
\left(1+\frac{5}{{{\mu_2}}^6}\right)
\nonumber\\[5pt]
&\quad<1-\frac{\pi^4}{16{{\mu_2}}^3}+\frac{3\pi^4}{16{{\mu_2}}^4}
-\frac{6}{{{\mu_2}}^5}+\frac{120+\frac{\pi^8}{256}}{{{\mu_2}}^6}.
\end{align}
Notice that
\begin{align*}
&\left(1+\frac{3\pi^4}{16{{\mu_2}}^4}+\frac{\pi^8}{35{{\mu_2}}^8}\right)
\left(1-\frac{\pi^4}{16{{\mu_2}}^3}-\frac{5\pi^8}{512{{\mu_2}}^7}\right)
\left(1-\frac{7}{{{\mu_2}}^5}-\frac{115}{{{\mu_2}}^6}\right)
\left(1-\frac{5}{{{\mu_2}}^6}\right)
\\[5pt]
&\quad-\left(1-\frac{\pi^4}{16{{\mu_2}}^3}+\frac{3\pi^4}{16{{\mu_2}}^4}
-\frac{7}{{{\mu_2}}^5}-\frac{130}{{{\mu_2}}^6}\right)
=\frac{1}{286720{{\mu_2}}^{27}}\sum_{i=0}^{21}c_i{{\mu_2}}^i,
\end{align*}
and
\begin{align*}
&\left(1+\frac{3\pi^4}{16{{\mu_2}}^4}+\frac{\pi^8}{12{{\mu_2}}^8}\right)
\left(1-\frac{\pi^4}{16{{\mu_2}}^3}+\frac{\pi^8}{256{{\mu_2}}^6}\right)
\left(1-\frac{6}{{{\mu_2}}^5}+\frac{115}{{{\mu_2}}^6}\right)
\left(1+\frac{5}{{{\mu_2}}^6}\right)
\\[5pt]
&\quad-\left(1-\frac{\pi^4}{16{{\mu_2}}^3}+\frac{3\pi^4}{16{{\mu_2}}^4}
-\frac{6}{{{\mu_2}}^5}+\frac{120+\frac{\pi^8}{256}}{{{\mu_2}}^6}\right)
=-\frac{1}{12288{{\mu_2}}^{26}}\sum_{j=0}^{20}d_j{{\mu_2}}^j,
\end{align*}
where $c_i$ and $d_j$ are real numbers. We list the values of $c_{19}-c_{21}$ and $d_{18}-d_{20}$:
\begin{align*}
&c_{19}=125440\pi^4+8192\pi^8,\quad\quad c_{20}=-6160\pi^8,\qquad\qquad c_{21}=2867200,
\\[3pt]
&d_{18}=-4608\pi^4-1024\pi^8,\quad\quad d_{19}=144\pi^8,\qquad\qquad\quad d_{20}=96\pi^8.
\end{align*}
It can be measured that when ${\mu_2}\geq3$, for $ 0\leq i\leq18$,
$$-|c_i|{{\mu_2}}^i\geq -|c_{19}|{{\mu_2}}^{19},$$
and when ${{\mu_2}}\geq3$, for $ 0\leq j\leq17$,
$$-|d_j|{{\mu_2}}^j\geq -|d_{18}|{{\mu_2}}^{18}.$$
Further, we can get
\begin{align*}
\sum_{i=0}^{21}c_i{{\mu_2}}^i
&\geq-\sum_{i=0}^{19}|c_i|{{\mu_2}}^{19}+c_{20}{{\mu_2}}^{20}+c_{21}{{\mu_2}}^{21}
\geq-20|c_{19}|{{\mu_2}}^{19}+c_{20}{{\mu_2}}^{20}+c_{21}{{\mu_2}}^{21}
\end{align*}
and similarly,
\begin{align*}
\sum_{j=0}^{20}d_j{{\mu_2}}^j
&\geq-\sum_{j=0}^{18}|d_j|{{\mu_2}}^{18}+d_{19}{{\mu_2}}^{19}+d_{20}{{\mu_2}}^{20}
\geq-19|d_{18}|{{\mu_2}}^{18}+d_{19}{{\mu_2}}^{19}+d_{20}{{\mu_2}}^{20}.
\end{align*}
Since for ${{\mu_2}}\geq38$,
$$-20|c_{19}|+c_{20}{{\mu_2}}+c_{21}{{\mu_2}}^{2}>0$$
and for ${{\mu_2}}\geq14$,
$$-19|d_{18}|+d_{19}{{\mu_2}}+d_{20}{{\mu_2}}^{2}>0,$$we see that (\ref{3.42}) and \eqref{3.43} are true.
By plugging \eqref{3.42} and \eqref{3.43} into \eqref{3.40} and \eqref{3.41}, respectively, we can establish that \eqref{Qkn-bounds} holds when $k=2$.
The proofs for other cases are similar, and therefore, we omit them here.
This completes the proof.
\end{proof}

Subsequently, we proceed with the proof of Theorem \ref{thm1.5}.

{\noindent {\it Proof of Theorem \ref{thm1.5}}}.
To establish log-concavity, we need to show that ${Q_k(n)}\leq1$ for $2\leq k\leq9$.
From Theorem \ref{thm4}, we see that for $2\leq k\leq9$ and ${n}\geq\tilde{n}_k$,
\begin{equation*}
{Q_k(n)}<\tilde{R}_k<1,
\end{equation*}
where $\tilde{n}_k$ and $\tilde{R}_k$ are defined in Table \ref{Qk(n)-bounds}.
Besides, ${Q_k(n)}<1$ holds as long as one can verify it holds for $\bar{n}_k\leq n\leq \tilde{n}_k$,
where $\bar{n}_k$ is defined in Theorem \ref{thm1.5}.
So ${Q_k(n)}<1$ for $n\geq\bar{n}_k$ and $\p_k(n)$ is log-concave for $n\geq\bar{n}_k$ when $2\leq k\leq9$.

To demonstrate that $\p_k(n)$ for $2\leq k\leq9$ satisfies the third order Tur\'{a}n inequalities for $n\geq\hat{n}_k$, we rely on Chen, Jia, and Wang's work \cite{turan-for-partition}. It is sufficient to prove that when $n\geq\hat{n}_k$,
\begin{align}\label{3.54}
4(1-{Q_k(n)})(1-Q_k(n+1))-(1-{Q_k(n)}Q_k(n+1))^2>0.
\end{align}
Recall that Jia \cite[Lemma~5.1]{Jia-2023} obtained the following result.
If $u$ and $v$ are positive real numbers satisfying $\frac{15}{16}\leq u<v<1$ and $u+\sqrt{(1-u)^3}>v$ holds, then
\begin{align}\label{lem3.4}
4(1-u)(1-v)-(1-uv)^2>0.
\end{align}
Using \eqref{lem3.4}, we come to prove that for $\mu_k\geq\tilde{n}_k$ and $2\leq k\leq9$,
\begin{equation}\label{3.499}
\frac{15}{16}\leq {Q_k(n)}<Q_k(n+1),
\end{equation}
and
\begin{equation}\label{3.49}
Q_k(n+1)<{Q_k(n)}+\sqrt{(1-{Q_k(n)})^3}.
\end{equation}
Next, we will prove the case when $k=2$.
By Theorem \ref{thm4}, we have for ${\mu_2}\geq167$, \[{Q_2(n)}>1-\frac{\pi^4}{16{{\mu_2}}^3}+\frac{3\pi^4}{16{{\mu_2}}^4}
-\frac{7}{{{\mu_2}}^5}-\frac{130}{{{\mu_2}}^6}.\]
It can be checked that when ${{\mu_2}}\geq4$,
\[1-\frac{\pi^4}{16{{\mu_2}}^3}+\frac{3\pi^4}{16{{\mu_2}}^4}
-\frac{7}{{{\mu_2}}^5}-\frac{130}{{{\mu_2}}^6}\geq\frac{15}{16}\]
Hence, for ${\mu_2}\geq167$,
\[{Q_2(n)}>\frac{15}{16}.\]
Moreover, by Theorem \ref{thm4}, it can be computed directly  that for ${{\mu_2}}\geq11$,
\begin{align*}
Q_2(n+1)-{Q_2(n)}>&\left(1-\frac{\pi^4}{16{{\mu_2^+}}^3}+\frac{3\pi^4}{16{{\mu_2^+}}^4}-\frac{7}{{{\mu_2^+}}^5}-\frac{130}{{{\mu_2^+}}^6}\right)\\
&-\left(1-\frac{\pi^4}{16{{\mu_2}}^3}+\frac{3\pi^4}{16{{\mu_2}}^4}-\frac{6}{{{\mu_2}}^5}+\frac{120+\frac{\pi^8}{256}}{{{\mu_2}}^6}\right)\\
=&\frac{\pi^4}{16}\left(\frac{1}{{{\mu_2}}^3}-\frac{1}{{{\mu_2^+}}^3}+\frac{3}{{{\mu_2^+}}^4}-\frac{3}{{{\mu_2}}^4}\right)\\
&+\frac{6}{{{\mu_2}}^5}-\frac{7}{{{\mu_2^+}}^5}-\frac{250+\frac{\pi^8}{256}}{{{\mu_2}}^6}\\
>&0.
\end{align*}
Here we realize (\ref{3.499}) is valid.
And we find that for ${\mu_2}\geq167$,
\begin{align}\label{3.50}
Q_2(n+1)-{Q_2(n)}&<\left(1-\frac{\pi^4}{16{{\mu_2^+}}^3}
+\frac{3\pi^4}{16{{\mu_2^+}}^4}-\frac{6}{{{\mu_2^+}}^5}
+\frac{120+\frac{\pi^8}{256}}{{{\mu_2^+}}^6}\right)
\nonumber\\[5pt]
&\quad -\left(1-\frac{\pi^4}{16{{\mu_2}}^3}+\frac{3\pi^4}{16{{\mu_2}}^4}
-\frac{7}{{{\mu_2}}^5}-\frac{130}{{{\mu_2}}^6}\right)
\nonumber\\[5pt]
&=\frac{\pi^4}{16}\left(\frac{1}{{{\mu_2}}^3}-\frac{1}{{{\mu_2^+}}^3}\right)
-\frac{3\pi^4}{16}\left(\frac{1}{{{\mu_2}}^4}-\frac{1}{{{\mu_2^+}}^4}\right)
\nonumber\\[5pt]
&\quad-\frac{6}{{{\mu_2^+}}^5}+\frac{7}{{{\mu_2}}^5}+\frac{120+\frac{\pi^8}{256}}
{{{\mu_2^+}}^6}+\frac{130}{{{\mu_2}}^6}.
\end{align}
Notice that when ${{\mu_2}}>0$,
\[\frac{1}{{{\mu_2}}^3}-\frac{1}{{{\mu_2^+}}^3}<\frac{3\pi^2}{4{{\mu_2}}^5},\]
and when ${{\mu_2}}\geq60$,
\[\frac{120+\frac{\pi^8}{256}}{{{\mu_2^+}}^6}
+\frac{130}{{{\mu_2}}^6}<\frac{288}{{{\mu_2}}^6}<\frac{\pi^2}{2{{\mu_2}}^5}.\]
Hence, we can rewrite \eqref{3.50} as
\begin{align*}
Q_2(n+1)-{Q_2(n)}&<\frac{\pi^4}{16}\cdot \frac{3\pi^2}{4{{\mu_2}}^5} + \frac{7}{{{\mu_2}}^5}+\frac{\pi^2}{2{{\mu_2}}^5}
\\[5pt]
&=\frac{\frac{3}{64}\pi^6+7+\frac{1}{2}\pi^2}{{{\mu_2}}^5}.
\end{align*}
By Theorem \ref{thm4},
\begin{align*}1-{Q_2(n)}&>\frac{\pi^4}{16{{\mu_2}}^3}-\frac{3\pi^4}{16{{\mu_2}}^4}+\frac{6}{{{\mu_2}}^5}-\frac{120+\frac{\pi^8}{256}}{{{\mu_2}}^6}.\end{align*}
Note for ${{\mu_2}}\geq73$,
\[\frac{1}{4{{\mu_2}}^3}>\frac{3\pi^4}{16{{\mu_2}}^4}-\frac{6}{{{\mu_2}}^5}+\frac{120+\frac{\pi^8}{256}}{{{\mu_2}}^6},\]so we have
\begin{align*}
1-{Q_2(n)}>\frac{\pi^4}{16{{\mu_2}}^3}-\frac{1}{4{{\mu_2}}^3}
=\frac{\pi^4-4}{16{{\mu_2}}^3}.
\end{align*}
It can be verified that for ${{\mu_2}}\geq17$,
$$\sqrt{\left(\frac{\pi^4-4}{16{{\mu_2}}^3}\right)^3}>\frac{\frac{3}{64}\pi^6+7+\frac{1}{2}\pi^2}{{{\mu_2}}^5}.$$
So (\ref{3.49}) holds for ${{\mu_2}}\geq167$. This prove the cases that for ${{\mu_2}}\geq167$.
When $65\leq n\leq5652$, we can compute that (\ref{3.54}) is also met. Thus for $n\geq65$, $\p_2(n)$ satisfies the third order Tur\'{a}n inequalities.
The proof for other cases is similar, and we omit it here. This completes the proof.
\qed

\vspace{0.5cm}
 \baselineskip 15pt
{\noindent\bf\large{\ Acknowledgements}} \vspace{7pt} \par
The second author would like to acknowledge that the research was supported by the National Natural Science Foundation of China (Grant Nos. 12001182 and 12171487),  the Fundamental Research Funds for the Central Universities (Grant No. 531118010411) and Hunan Provincial Natural Science Foundation of China (Grant No. 2021JJ40037).


\begin{thebibliography}{99}
\setlength{\itemsep}{-.8mm}
\addcontentsline{toc}{section}{References}

\bibitem{A-A-M-COMBI}
A.A. Alanazi, S.M.III Gagola and A.O. Munagi,
Combinatorial proof of a partition inequality of Bessenrodt-Ono,
Ann. Comb. 21 (3) (2017) 331--337.

\bibitem{Beckwith-Bessenrodt-2016}
O. Beckwith and C. Bessenrodt, Multiplicative properties of the number of $k$-regular partitions, Ann. Comb. 20 (2) (2016) 231--250.

\bibitem{Bessenrodt-Ono-2016}
C.~Bessenrodt and K.~Ono, Maximal multiplicative properties of partitions, Ann. Comb. 20 (1) (2016) 59--64.

\bibitem{bringmann-i-bessel}
K. Bringmann, B. Kane, L. Rolen and Z. Tripp, Fractional partitions and conjectures of Chern-Fu-Tang and Heim-Neuhauser, Trans. Amer. Math. Soc. Ser. B 8 (2021) 615--634.

\bibitem{Chen-2017}
W.Y.C. Chen, The spt-function of Andrews, Surveys in Combinatorics 2017, 141--203, London Math. Soc. Lecture Note Ser., 440, Cambridge Univ. Press, Cambridge, 2017.

\bibitem{turan-for-partition}
W.Y.C. Chen, D.X.Q. Jia and L.X.W. Wang, Higher order Tur\'{a}n inequalities for the partition function, Trans. Amer. Math. Soc. 372 (3) (2019) 2143--2165.

\bibitem{chern-eta}
S. Chern, Asymptotics for the Fourier coefficients of eta-quotients, J. Number Theory 199 (2019) 168--191.

\bibitem{Dawsey-Masri-2019}
M.L. Dawsey and R. Masri, Effective bounds for the Andrews spt-function, Forum Math. 31 (3) (2019) 743--767.

\bibitem{DeSalvo-Pak-2015}
S.~DeSalvo and I.~Pak, Log-concavity of the partition function, Ramanujan J. 38 (1) (2015) 61--73.

\bibitem{dong-ji}
J.J.W. Dong and  K.Q. Ji, Higher order Tur\'{a}n inequalities for the distinct partition function, arXiv:2303.05243.

\bibitem{Engel-2017}
B.~Engel, Log-concavity of the overpartition function, Ramanujan J. 43 (2) (2017) 229--241.

\bibitem{Griffin-Ono-Rolen-Zagier-2019}
M. Griffin, K. Ono, L. Rolen and D. Zagier, Jensen polynomials for the Riemann zeta function and other sequences, Proc. Natl. Acad. Sci. USA 116 (23) (2019) 11103--11110.

\bibitem{Jia-2023}
D.X.Q. Jia,  Inequalities for the broken $k$-diamond partition functions, J. Number Theory 249 (2023) 314--347.

\bibitem{Li-2023}
X. Li, Polynomization of the Liu-Zhang inequality for the overpartition function, Ramanujan J. DOI:10.1007/s11139-023-00711-7.

\bibitem{Liu-Zhang-2021}
E.Y.S.~Liu and H.W.J.~Zhang, Inequalities for the overpartition function, Ramanujan J. 54 (3) (2021) 485--509.

\bibitem{Lovejoy-2003}
J. Lovejoy, Gordon's theorem for overpartitions, J. Combin. Theory Ser. A 103 (2) (2003) 393--401.

\bibitem{Ono-Ken-Rolen-2022}
K. Ono, S. Ken and L. Rolen, Tur\'{a}n inequalities for the plane partition function, Adv. Math. 409 (2022) part B, Paper No. 108692, 31 pp.

\bibitem{shen-l-re-overpar}
E.Y.Y. Shen, Arithmetic properties of $l$-regular overpartitions, Int. J. Number Theory 12 (3) (2016) 841--852.



\end{thebibliography}
\end{document}